\documentclass[a4paper,10pt,reqno]{amsart}
\usepackage{epsf,epsfig}
\usepackage{graphics,color}
\usepackage{amsmath}
\usepackage{amssymb}
\usepackage{cite}
\usepackage{verbatim}
\usepackage{float}
\usepackage{graphicx}
\usepackage{amsthm}
\usepackage{textcomp}
\usepackage{subfig}
\usepackage{enumitem} 

\usepackage{amstext}
\usepackage{amsfonts}

\usepackage{ifthen}
\usepackage{stmaryrd}
\usepackage{algorithm}
\usepackage{algpseudocode}
\usepackage{amscd}
\usepackage{pgfplots}
\usepackage{pgf}
\usepackage{tikz}

\usepackage[nice]{nicefrac}

\usepackage{varioref}

\newcommand{\mat}[1]{\ensuremath{\begin{pmatrix}#1\end{pmatrix}}}	

\usepackage{chngcntr}
\counterwithin{figure}{section}
\counterwithin{table}{section}

\newcommand{\RR}{\mathbb{R}}

\newcommand{\Om}{\Omega}

\newcommand{\SL}{\mathrm{SL}(d)}

\def\NN{\mathbb N}

\def\tr{\,{\rm tr\,}}

\def\cof{\,{\rm cof\,}}

\def\det{\,{\rm det\,}}

\def\argmin{\,{\rm argmin}}

\def\d{\,\mathrm{d}}

\newcommand{\We}{W_{\mathrm{e}}}
\newcommand{\Wp}{W_{\mathrm{p}}}
\newcommand{\bA}{\boldsymbol{A}}
\newcommand{\bF}{\boldsymbol{F}}
\newcommand{\bFe}{\boldsymbol{F}_{\mathrm{e}}}

\newcommand{\bP}{\boldsymbol{P}}
\newcommand{\bX}{\boldsymbol{X}}

\newcommand{\bsigma}{\boldsymbol{\sigma}}
\newcommand{\bI}{\boldsymbol{I}}

\newcommand{\bQ}{\boldsymbol{Q}}

\newcommand{\bB}{\boldsymbol{B}}

\newcommand{\bu}{\boldsymbol{u}}
\newcommand{\bv}{\boldsymbol{v}}
\newcommand{\by}{\boldsymbol{y}}
\newcommand{\bx}{\boldsymbol{x}}

\newcommand{\bn}{\boldsymbol{n}}
\renewcommand{\by}{\boldsymbol{y}}
\newcommand{\bg}{\boldsymbol{g}}
\newcommand{\bforce}{\boldsymbol{f}}
\newcommand{\bnabla}{\nabla}
\newcommand{\bphi}{\boldsymbol{\varphi}}
\newcommand{\bvarphi}{\boldsymbol{\varphi}}
\newcommand{\mup}{\mu_{\rm p}}
\newcommand{\muz}{\mu_{\rm z}}

\newcommand{\qe}{q_{\rm e}}

\newcommand{\Hone}[1][]{\ensuremath{H^1\ifthenelse{\equal{#1}{}}{}{(#1)}}}
\newcommand{\Ltwo}[1][]{\ensuremath{L^2\ifthenelse{\equal{#1}{}}{}{(#1)}}}

\numberwithin{equation}{section}
\mathchardef\emptyset="001F

\newtheorem{theorem}{Theorem}[section]

\theoremstyle{definition}
\newtheorem{remark}[theorem]{Remark}

\usepackage{listings}
\usepackage{xcolor}

\definecolor{codegreen}{rgb}{0,0.6,0}
\definecolor{codegray}{rgb}{0.5,0.5,0.5}
\definecolor{codepurple}{rgb}{0.58,0,0.82}
\definecolor{backcolour}{rgb}{0.7,0.7,0.7}
\definecolor{light-gray}{gray}{0.9}

\lstdefinestyle{mystyle}{
	backgroundcolor=\color{light-gray},
	commentstyle=\color{codegreen},
	keywordstyle=\color{magenta},
	numberstyle=\tiny\color{codegray},
	stringstyle=\color{codepurple},
	basicstyle=\footnotesize,
	breakatwhitespace=false,         
	breaklines=true,                 
	captionpos=b,                    
	keepspaces=true,                 
	numbers=left,                    
	numbersep=5pt,                  
	showspaces=false,                
	showstringspaces=false,
	showtabs=false,                  
	tabsize=3,
	frame=single
}

\title[A finite-strain model for incomplete damage in elastoplastic materials]{A finite-strain model for incomplete damage in elastoplastic materials}
\date{\today}

\author{David Melching} 
\address[David Melching]{Faculty of Mathematics, University of Vienna, Oskar-Morgenstern-Platz 1, 1090 Wien, Austria.}
\email{david.melching@univie.ac.at}
\urladdr{http://www.mat.univie.ac.at/~melching/}

\author{Michael Neunteufel} 
\address[Michael Neunteufel]{Institute for Analysis and Scientific Computing, TU Wien, Wiedner Hauptstrasse 8-10, 1040 Wien, Austria.}
\email{michael.neunteufel@tuwien.ac.at}
\urladdr{https://www.asc.tuwien.ac.at/~schoeberl/wiki/index.php/Michael\_Neunteufel}

\author{Joachim Sch\"oberl} 
\address[Joachim Sch\"oberl]{Institute for Analysis and Scientific Computing, TU Wien, Wiedner Hauptstrasse 8-10, 1040 Wien, Austria.}
\email{joachim.schoeberl@tuwien.ac.at}
\urladdr{https://www.asc.tuwien.ac.at/~schoeberl/wiki/index.php/Joachim\_Sch{\"o}berl}

\author{Ulisse Stefanelli} 
\address[Ulisse Stefanelli]{Faculty of Mathematics, University of
  Vienna, Oskar-Morgenstern-Platz 1, 1090 Wien, Austria,
Vienna Research Platform on Accelerating Photoreaction Discovery, University of Vienna, W\"ahringer Str. 17, 1090 Vienna, Austria,
  and Istituto di Matematica
Applicata e Tecnologie Informatiche {\it E. Magenes}, v. Ferrata 1, 27100 Pavia, Italy.}
\email{ulisse.stefanelli@univie.ac.at}
\urladdr{mat.univie.ac.at/~stefanelli}

\begin{document}
\counterwithin{lstlisting}{section}
	
\maketitle

\begin{abstract}
We address a three-dimensional model capable of describing coupled
damage and plastic effects in solids at finite strains. Formulated
within the variational setting of {\it generalized standard materials}, the constitutive model results from the balance of conservative and dissipative forces. Material response is rate-independent and associative and damage evolution is unidirectional. We assess the model features and performance on both uniaxial and non-proportional biaxial tests. 

The constitutive model is then complemented with the quasistatic equilibrium system and initial and boundary conditions. We produce numerical simulations with the help of the powerful multiphysics finite element software NETGEN/NGSolve. We show the flexibility of the implementation and run simulations for various 2D and 3D settings under different choices of boundary conditions and possibly in presence of pre-damaged regions.\newline

\end{abstract}

\section*{Introduction}\label{Sec:Intro}

Damage in elastoplastic material is induced by the onset and development of microvoids and microcracks under deformation. In ductile materials, these often result from the rearranging of crystalline structures along slip planes, as well as by the initiation and the evolution of dislocation structures \cite{LC90}. In a continuum theory, such microscopic phenomena are averaged out and one focuses on a macroscopic phenomenological descriptor $z\in [0,1]$, indicating a pointwise measure of soundness of the material. Reflecting the {\it coupled} evolution of elastoplasticity and damage, the damage parameter $z$ is regarded as an internal variable and is coupled with deformation and plastic strain in the constitutive model for the material. This approach dates back to Kachanov~\cite{Kachanov,Kachanov2} and has been successfully developed
in different settings. The Reader is referred to Chaboche~\cite{Chaboche,Chaboche2_1,Chaboche2_2}, Krajdnovic~\cite{Krajdnovic}, Lemaitre~\cite{Lamaitre,Lamaitre2}, and Murakami~\cite{Murakami}, among others.

In the infinitesimal deformation regime, the engineering literature on coupled damage and plasticity models is vast. Correspondingly, the mathematical treatment of small-strain coupled models has attracted a good deal of attention, see \cite{Bonetti,Crismale,DalMaso,Davoli,MR15,Roub16} as well as the
recent monograph \cite{Kruzik-Roubicek} for a collection of results. Numerical methods for small strain elastoplasticity have also been intensively investigated, see \cite{BCV05,Carst97,Cerm19,Glow81,Grub08,Han95,John77,John78,Nagt74,SH98}. A-posteriori error estimators in space and time together with adaptive mesh refinement strategies to recover optimal convergence rates have been explored \cite{Albert99,Gall96,Peric94,Ramm03}, and high order polynomial methods have been developed \cite{Dust01,Szabo04}.

In the finite-strain setting, the modeling literature on damage in elastoplastic materials is less abundant. Seminal contributions in this direction are due to Ju~\cite{Ju} and Simo~and~Ju~\cite{Simo}. Steinmann, Miehe, and Stein \cite{Steinmann} presented a finite-strain reformulation of Lemaitre~\cite{Lamaitre} and Gurson~\cite{Gurson} models and Mahnken~\cite{Mahnken} extended the finite-plasticity model by Miehe~\cite{Miehe} in order to include damage. More recently, anisotropic behavior has been investigated in \cite{Ganjiani,Menzel}. Algorithms and general strategies to solve finite elastoplastic constitutive models can be found, e.g., in \cite{MIEHE94,MR16,Simo85}. 

From the mathematical standpoint, rigorous existence and approximability results in the finite-elastoplastic setting are just a few. The static problem has been firstly addressed by Mielke \cite{Mielke04b} under restrictions on the choice of the driving functionals, by  Mielke and M\"uller \cite{Mielke-Mueller} under a gradient-type regularization for the plastic strain, and in \cite{compos} in the dislocation-free case. The quasistatic evolution problem has been analyzed by Mainik and Mielke \cite{MM09} under gradient regularization for the plastic strain, see also  \cite{cplas_part1,cplas_part2} for a similar result in a symmetric case, and in \cite{compos2} in the dislocation-free case. The Reader is also referred to \cite{Davoli14,Davoli14b,Giacomini13,MS13,curl} for rigorous linearization results for small strains. 

A first existence result for the coupled quasistatic evolution for finite elastoplasticity and damage has been recently obtained in \cite{MSZ19}. The aim of this paper is to elucidate the mechanical basis of such model and present a suite of numerical simulations. The model is formulated within the classical framework of {\it generalized standard materials} \cite{HN75}. The state of the material system is described in terms of the {\it deformation} $\bvarphi$ from its reference configuration, the {\it plastic strain} $\bP$, and the scalar {\it damage} parameter $z$. As it is customary in finite plasticity, we assume that the deformation gradient $\bF:=\bnabla \bvarphi$ is multiplicatively decomposed as $ \bF = \bFe \bP$ \cite{Kroener, Lee69} where $\bFe$ is the {\it elastic} strain. We assume {\it hyperelastic, isotropic} material response as well as {\it kinematic hardening} dynamics for plasticization. Additionally, we impose $\det \bP=1$ in order to model the possible {\it isochoric} nature of plastic effects \cite{SH98}. 

The damage parameter $z$ takes values in $[0,1]$, where $z=1$ stands for the undamaged material and $z=0$ represents the maximally damaged case. We assume damage dynamics to be {\it isotropic}, {\it unidirectional}, meaning that once damaged the material cannot heal, and {\it incomplete} \cite{LC90}, which entails that even at $z=0$ the material can still sustain stresses. This is for instance the case of a composite material where a damageable component is included in a non-damageable matrix. 

The material constitutive model results from the interplay of energy-storage and dissipation mechanisms \cite{Maugin92}. These are specified by prescribing the {\it total stored energy}
\begin{equation}
W(\bF,\bP,z) = \zeta(z)\widehat \We(\bF \bP^{-1}) + \Wp(\bP)\label{stored0}
\end{equation}
resulting from the sum of the elastic and the plastic (hardening) energies. As for the {\it total dissipation potential} we assume
\begin{equation}
R(\bP,z,\dot \bP, \dot z) = R_{\rm z}(\dot z) + \rho(z)R_{\rm p}(\bP,\dot
\bP) \label{dissi}
\end{equation}
so that damage and plastic dissipative effects are combined via an {\it associative} flow rule. We refer to Section \ref{Sec:Model} for all details and limit ourselves here in observing that the damage parameter $z$ influences the actual elastic response of the medium and the effective plastic yield stress via the monotone {\it coupling functions} $\zeta$ and $\rho$. This corresponds to the case of an elastoplastic material  weakening upon damage accumulation. The material constitutive model is then obtained from the balance of {\it conservative} and {\it dissipative} forces as 
$$\partial_{(\bF,\bP,z)} W (\bF,\bP,z) + \partial_{(\dot \bF,\dot \bP,\dot z)}R (\bP,z,\dot \bP, \dot z) \ni 0,$$
see Subsection \ref{sec:final}. This balance embodies the variational structure of the model, which in turn yields thermodynamical consistency, see Subsection \ref{sec:flow}, and allows for a comprehensive mathematical treatment \cite{MSZ19}. The potentials $R_{\rm z}$ and $R_{\rm p}$ are positively $1$-homogeneous in the rate variables, so that the constitutive model is of {\it rate-independent} type \cite{MR15}, as it is often assumed in damage and plasticity.

We provide a full account of the model derivation in Section
\ref{Sec:Model}. We specify and discuss a choice for
potentials and functions in \eqref{stored0}-\eqref{dissi}, depending
on a minimal set of material parameters. Such choice is then tested on
both uniaxial and biaxial tests in order to illustrate the performance
of the model in consistently reproducing damage and plasticity
coupling and to highlight the role of the material parameters. In
particular, we introduce a semi-implicit time-discretization of the 
constitutive  model.

The constitutive model is then complemented with the quasistatic
equilibrium system in Section \ref{sec:quasi}.  By specifying
boundary conditions, we present an existence result for quasistatic
incremental solutions in Subsection \ref{Subsec:existence}. Details on
regularization and discretization of such problem via conformal finite
elements are also provided.

In Section~\ref{Sec:Examples}, we present a collection of numerical
experiments in discrete time. Approximate incremental problems are solved   
numerically within the frame of the NETGEN/NGSolve finite element
library \footnote{www.ngsolve.org} \cite{Sch97,Sch14}. This provides a
flexible and effective environment for testing different specimen
geometries and material settings.  Numerical simulations both in
two and three space dimensions  are  presented. Examples and code snippets are also provided and some algorithmical aspects related to mesh generation are discussed in the Appendix.

\section{Constitutive model}\label{Sec:Model}

We devote this section to the introduction of the constitutive material model. As mentioned, this fits into the classical frame of {\it generalized standard materials} \cite{HN75}, being governed by the interplay of energy-storage and energy-dissipation mechanisms. In the following, we introduce energy and dissipation separately in Subsections \ref{sec:energy} and \ref{sec:dissipation} and then combine them via flow rules in Subsection \ref{sec:flow}. Eventually, the final form of the constitutive model is given in Subsection \ref{sec:final}. 

The model performance is assessed in Subsection \ref{Subsec:ParameterAnalysis} where we present the results of uniaxial and biaxial tests and we discuss the role of the various material parameters by presenting some numerical simulations. These simulations are based on a semiimplicit time-discrete scheme, which is introduced in Subsection \ref{sec:time-discrete-frame}. Some comments on the algorithmical implementation of the scheme are provided in Subsection \ref{Subsec:ngsolve}.

Let us start by fixing some notation. Throughout the paper, we denote by $d =2,3$ the dimension of the body. Vectors $\boldsymbol v \in \RR^d$ are indicated in boldface and we use the standard notation for the product $\bv \cdot \boldsymbol w = \sum_i v_iw_i$. Matrices $\bA \in \RR^{d\times d}$ are denoted by boldfaced capital letters and are endowed with the product $\bA : \bB = \sum_{i,j}A_{ij}B_{ij}$ and the corresponding norm $|\bA|^2 := \bA: \bA$. We indicate by $\SL$ the special linear group, i.e., matrices $\bA \in \RR^{d\times d}$ satisfying $\det \bA =1$.

We recall some basic notions from convex analysis. Given the convex function $f: \RR^n \to \RR$, a vector $\bv \in \RR^n$ is called {\it subgradient} of $f$ at $\bx \in \RR^n$, if for every $\by \in \RR^n$ one has that $f(\by)-f(\bx) \ge \bv \cdot (\by-\bx)$. The {\it subdifferential} $\partial f(\bx)$ is  the set of all subgradients of $f$ at $\bx$. The {\it Fenchel conjugate} $f^*:\RR^n \to \RR$ is defined as  $f^*(\bv):= \sup \{ \bv\cdot \bx - f(\bx)\, :\,  \bx \in \RR^n \}$ and one has that $\bx \in \partial f^*(\bv)$ iff $ \bv \in \partial f(\bx),$ see, e.g., \cite{R70} for details.

\subsection{Energy}\label{sec:energy} 
As mentioned, the state of the material is described by the {\it strain gradient} $\bF \in \RR^{d\times d}$,  the {\it plastic strain} $\bP\in \SL$, and the {\it damage parameter} $z\in [0,1]$. Recall that $z=1$ corresponds to the undamaged material whereas $z=0$ represents the fully damaged situation. The energy of the material is assumed to be additively decomposed into a purely elastic part and a purely plastic part, see \eqref{stored0}. The former is a function of the elastic strain $\bF \bP^{-1}$ and the latter is a  function of the plastic strain $\bP$. In particular, we introduce the {\it stored energy}
\begin{equation}
W(\bF,\bP,z) = \We(\bF \bP^{-1},z) + \Wp(\bP).\label{stored}
\end{equation}
Here, $\We:\RR^{d\times d} \times [0,1] \to \RR_+$ denotes the {\it elastic energy} whereas $\Wp: \SL \to \RR_+$ represents the {\it kinematic hardening} potential. 

Damage influences the elastic response of the material by modifying the elastic energy density. In particular, as damage accumulates (i.e., $z$ decreases) the elastic energy decreases. This corresponds to the fact that the material weakens upon damaging. This behavior is modeled via the structural assumption
\begin{equation}
\We(\bF \bP^{-1},z) = \zeta(z) \widehat \We(\bF \bP^{-1}),
\end{equation}
where $\widehat \We:\RR^{d\times d} \to \RR_+$ and $\zeta: \RR \to
\RR_+$ is a continuous, positive, and monotone-increasing {\it energy-degradation} function. Specifically, we 
assume $\widehat \We$ to be of {\it Neo--Hookean} type, namely,
\begin{align}
\label{eq:neo_hooke_en}
\widehat \We(\bFe):= \frac{\mu}{2} |\bFe|^2 + \Gamma(\det\bFe), \quad \text{with } \Gamma(\delta) = - \mu\ln(\delta) + \frac{\lambda}{2}(\delta-1)^2 - \frac{\mu d}{2},
\end{align}
for given Lam\'e parameters $\mu > 0$ and $\lambda \ge -d\mu/2$. Note that this choice complies with {\it frame indifference}, i.e., $\widehat \We(\bQ  \bFe ) = \widehat \We(\bFe)$ for all $\bFe \in {\mathbb R}^{d\times d}$ and all rotations $\bQ$ with $\det \bQ =1$  and {\it local non-interpenetrability of matter}, i.e. $\We(\bFe) \to +\infty$ as $\det\bFe \to 0^+$.  Note also that $\We(\bI)=D\We(\bI)=0$ by the choice of $\Gamma$. 
The energy-degradation function $\zeta$ is prescribed as
\begin{equation}\label{def:zeta}
\zeta(z) :=
\zeta_0+(1-\zeta_0)(z^+)^2
\end{equation}
for some $\zeta_0 \in (0,1]$, where $z^+ = \max \{0,z\}$ denotes the {\it positive part}.

As for the kinematic hardening potential $W_{\rm p}$ we assume
\begin{align}
\label{eq:hardening_en}
\Wp(\bP):=\begin{cases}
\displaystyle \frac{H}{2}\, |\bP-\bI|^2, &\text{if }\bP \in K \\
+ \infty, &\text{else.}
\end{cases}
\end{align}
where $H>0$ is the {\it hardening parameter} and  $K$ is compact
in $\SL$ and contains a neighborhood of
$\bI$. The set $K$ acts as constraint on $\bP$, ensuring that we can find a constant $c_K>0$ such that 
\begin{align}
  &\bP \in K \ \ \Rightarrow \ \ |\bP|  \leq
    c_K. \label{conscomp}
\end{align}
Such uniform bounds expedite the existence proof of Section
\ref{Subsec:existence}. However, by computing variations of $\Wp$ in the coming
sections and in computations, the constraint $\bP\in K$ will be
systematically neglected, for it can always be assumed to be inactive
by choosing $K$ large enough.

\subsection{Constitutive relations} \label{sec:consti} 
By computing the variations of the energy with respect to states we obtain the constitutive relations 
\begin{align}
\bsigma &= \frac{\partial W}{\partial \bF} = \zeta(z) D\widehat \We(\bF \bP^{-1}) \bP^{-\top}, \label{stress} \\
\bX &= \frac{\partial W}{\partial \bP} = -\zeta(z)\bP^{-\top}\bF^\top D\widehat \We(\bF \bP^{-1}) \bP^{-\top} + D\Wp(\bP) = -\bP^{-\top}\bF^\top \bsigma + D\Wp(\bP),\\
\xi &= \frac{\partial W}{\partial z} = \zeta'(z) \widehat \We(\bF \bP^{-1}). \label{xi}
\end{align}
Here, $\bsigma$ is the \textit{first Piola-Kirchhoff stress tensor}, whereas $\bX$ and $\xi$ are the {\it thermodynamic driving forces} driving the evolution of the plastic strain $\bP$ and damage $z$, respectively.

Moving from \eqref{eq:neo_hooke_en}-\eqref{eq:hardening_en}, we readily compute that 
\begin{align*}
D\widehat{\We}(\bFe) &= \mu \bFe + \Gamma'(\det\bFe)\cof\bFe= \mu(\bFe-\bFe^{-\top}) +\lambda  (\det\bFe-1)\cof\bFe, \\
D\Wp(\bP) &= H(\bP-\bI), \\ 
\zeta'(z) &= 2(1-\zeta_0)z^+.
\end{align*}
Hence, the constitutive relations \eqref{stress}-\eqref{xi} can be specified as
\begin{align}
\bsigma &= \big(\zeta_0 + (1-\zeta_0)(z^+)^2\big) \left(\mu(\bFe-\bFe^{-\top}) +\lambda (\det\bFe-1)\cof\bFe \right) \bP^{-\top} , \label{stress1} \\
\bX &= -\bFe^{\top} \bsigma +H(\bP-\bI),\\
\xi &= \big(2(1-\zeta_0)z^+\big)\left( \frac{\mu}{2} |\bFe|^2 + \Gamma(\det\bFe) \right). \label{xi1}
\end{align}

\subsection{Dissipation} \label{sec:dissipation} 
Both damage and plasticization contribute to dissipation. These two dissipative mechanisms are described by the {\it dissipation potentials} $R_{\rm z}$ and $R_{\rm p}$, interacting via a positive and  monotone-increasing {\it dissipation-coupling} function $\rho: {\mathbb R}\to {\mathbb R}_+$ of damage. In particular, we prescribe the {\it total dissipation potential} to have the form
\begin{equation}
\label{eq:dissipation_plast_dam}
R(\bP,z,\dot \bP,\dot z) = R_{\rm z}(\dot z) + \rho(z) R_{\rm p}(\bP,\dot \bP).
\end{equation}
Note that the total dissipation potential depends on the current states $\bP$ and $z$, as well as on the rates $\dot \bP$ and $\dot z$.

The {\it plastic dissipation} potential $R_{\rm p}: \SL \times \RR^{d\times d} \to [0,\infty],$ is assumed to be positively 1-homogeneous in its second component, i.e. $R_{\rm p}(\bP,\lambda \dot \bP) = \lambda R_{\rm p}(\bP,\dot \bP)$ for all $\lambda \ge 0$. In addition, we require {\it plastic indifference}, namely that dissipation is independent from prior plasticization. This translates in asking $R_{\rm p}(\bP\bQ,\dot\bP\bQ)=R_{\rm p}(\bP,\dot \bP)$ for every $\bQ \in \SL$. These properties imply that 
$$R_{\rm p}(\bP,\dot \bP) = R_{\rm p}(\bI,\dot \bP\bP^{-1})=: \widehat R_{\rm p}(\dot \bP \bP^{-1})$$ 
for some given positively 1-homogeneous function $\widehat R_{\rm p}:\RR^{d\times d} \to [0,\infty]$. We choose $\widehat R_{\rm p}$ of Von Mises-type, namely,
\begin{equation}\label{mises}
\widehat R_{\rm p}(\bA) 
 = \begin{cases}
\sigma_{\rm p} |\bA|, & \text{if }\tr \bA = 0,\\
\infty, & \text{else,}
\end{cases}
\end{equation} 
where $\sigma_{\rm p} > 0$ denotes the {\it plastic yield stress}. The constraint on the trace entails that plastic evolution is isochoric: starting from an initial value $\bP_0 \in \SL$, the whole evolution takes place in $\SL$, for the space of trace-free matrices is the tangent space to $\SL$ at the identity. Therefore, at all times, plastic-strain rates are tangent to $\SL$. In the numerical simulations below the isochoric constraint is enforced by a more direct approach, see \eqref{eq:fes_plastic_strains} in Subsection~\ref{sec:FEM}.

As for the {\it damage dissipation potential} $R_{\rm z}: \RR \to [0,\infty]$ we let
\begin{equation}
\label{eq:damage_dissipation}
R_{\rm z}(\dot z):=\begin{cases}
\sigma_{\rm z} |\dot z| & \text{ if } \dot z \le 0, \\
\infty & \text{ else.}
\end{cases}
\end{equation}
The parameter $\sigma_{\rm z} >0$ can be understood as a {\it damage yield stress}. By setting $R_{\rm z}(\dot z)=\infty$ for $\dot z > 0$, damage is constrained to be {\it unidirectional}, i.e., no healing of the material is allowed. 

The monotone-increasing dissipation-coupling function $\rho$ is defined as 
\begin{equation}\label{def:rho}
\rho(z) :=
\rho_0+(1-\rho_0)(z^+)^2
\end{equation}
for some $\rho_0 \in (0,1]$. Referring to \eqref{eq:dissipation_plast_dam}-\eqref{mises}, one realizes that increasing damage induces a reduction the {\it effective} plastic yield stress 
\begin{equation}\label{effe}
\big(\rho_0+(1-\rho_0)(z^+)^2\big)\sigma_{\rm p}
\end{equation}
of the material. Note that the limiting case of {\it damage-free} finite plasticity can be included in this
description by letting $\zeta_0 = \rho_0 = 1$.

As damage is unidirectional, at all times $z$ is bounded from above by the initial damage $z_0\in [0,1]$. As such, the functions $\zeta$ and $\rho$ actually need to be specified on $(-\infty,1]$ only. Defining these functions on the whole real line and, in particular, assuming them to be constant on the semiline $(-\infty,0]$ is advantageous from the algorithmical viewpoint. Indeed, the constraint $z \in [0,1]$ is not directly imposed as on the one hand $z\leq 1$ follows from the monotonicity of $z$ in time and on the other hand $0\leq z$ ensues from the dissipation minimality, see Subsection \ref{sec:time-discrete-frame} below.

\subsection{Flow rules}\label{sec:flow}
The {\it internal} variables $\bP$ and $z$ are assumed to evolve according to classical normality rules. In particular, the evolution of the plastic strain $\bP$ is driven by the thermodynamic force $\bX$ and follows
\begin{equation}\label{force_balance_P}
-\bX \in \partial_{\dot \bP} R(\bP,z,\dot \bP,\dot z),
\end{equation}
where $\partial_{\dot \bP}$ denotes the subdifferential with respect to the rate $\dot \bP$. 

The damage parameter $z$ is driven by the thermodynamic force $\xi$ via the inclusion
\begin{equation}\label{force_balance_z}
-\xi \in \partial_{\dot z} R(\bP,z,\dot \bP,\dot z).
\end{equation}
The constitutive model is hence {\it associative}, for the evolution direction is determined by the gradient of the total dissipation potential with respect to rates.

In combination with the constitutive relations \eqref{stress}-\eqref{xi}, the flow rules \eqref{force_balance_P}-\eqref{force_balance_z} entail the thermodynamical consistency of model. Indeed, by assuming smoothness one readily checks the {\it Clausius-Duhem inequality}
\begin{align*}
  &\frac{\d}{\d t} W(\bF,\bP,z) - \bsigma: \dot \bF =
  \left(\frac{\partial W}{\partial \bF} - \bsigma\right):\dot \bF +
  \frac{\partial W}{\partial \bP}:\dot \bP + \frac{\partial
  W}{\partial z}\dot z \\&\qquad = \bX :\dot \bP + \xi \dot z =- R(\bP,z,\dot
  \bP,\dot z)\leq 0.
\end{align*}

\subsection{Final form of the constitutive material model}\label{sec:final}
Let us collect here the constitutive material model resulting from the constitutive relations \eqref{stress}-\eqref{xi} and the flow rules \eqref{force_balance_P}, \eqref{force_balance_z}, namely,
\begin{align}
  & \zeta(z) \left(\mu(\bFe-\bFe^{-\top}) +\lambda (\det\bFe-1)\cof\bFe \right) \bP^{-\top} = \bsigma, \label{stress0} 
\\
&\rho(z)\partial 
 \widehat R_{\rm p}( \dot \bP \bP^{-1})\bP^{-\top}
  -\bFe^{\top} \bsigma + H(\bP - \bI) \ni {\boldsymbol 0},\label{P0}
\\
&\partial R_{\rm z}(\dot z)+ \zeta'(z)\left( \frac{\mu}{2} |\bFe|^2 + \Gamma(\det\bFe) \right) \ni 0.\label{xi0}  
\end{align}
The system \eqref{stress0}-\eqref{xi0} has to be complemented with initial conditions for $\bP$ and $z$, namely 
\begin{equation}
\label{eq:init_cond}
(\bP(0),z(0)) = (\bP_{0}, z_0),  
\end{equation}
for some suitably given  initial values for plastic strain and damage $\bP_0 \in \SL$ and $z_0 \in [0,1]$.

\subsection{Time-incremental formulation}\label{sec:time-discrete-frame}
We comment now on the formulation of the model in discrete time, which in particular delivers a time-discrete approximation of the time-continuous model from Subsection~\ref{sec:final}.

Let $\bsigma: [0,T] \to \RR^{d\times d}$ be a prescribed stress history and  $ (\bP_{0}, z_0)\in \SL \times [0,1]$ be given. We consider a uniform partition $\Pi = \{ 0=t_0 < t_1 < \dots < t_{N-1} < t_N=T \}, N \in \NN, t_i=i\tau , \tau >0$, of the time interval $[0,T]$ (non-uniform partitions can be considered as well). Approximate solutions to \eqref{stress0}-\eqref{xi0} can be constructed by successively solving the following incremental minimization problems: 

Given $(\bP_{i-1}, z_{i-1})$, find
\begin{align}\label{inc_min_0D}
(\bF_i,\bP_i,z_i) \in \argmin \bigg\{ &\zeta(z)\widehat \We(\bF \bP^{-1}) + \Wp(\bP) - \bsigma_i:\bF \nonumber \\ &+R_{\rm z}(z-z_{i-1}) + \rho(z_{i-1}) \widehat R_{\rm p}(\bP \bP_{i-1}^{-1}-\bI) \bigg\},
\end{align}
where the minimum is taken over $(\bF,\bP,z) \in \RR^{d\times d} \times \SL \times [0,z_{i-1}]$ and $\bsigma_i := \bsigma(t_i)$. Indeed, a straight-forward calculation reveals that the Euler-Lagrange equations associated to the minimization problem \eqref{inc_min_0D} are given by
\begin{align}
  & \zeta(z_i) \left(\mu((\bFe)_i-(\bFe)_i^{-\top}) +\lambda (\det(\bFe)_i-1)\cof(\bFe)_i \right) \bP^{-\top}_i = \bsigma_i, \label{stress0d} 
\\
&   \rho(z_{i-1}) \partial \widehat R_{\rm p}(\bP_i\bP_{i-1}^{-1}-\bI)\bP^{-\top}_{i-1}
  -(\bFe)^\top_i \bsigma_i + H(\bP_i - \bI) \ni {\boldsymbol 0},\label{P0d}
\\
&\partial R_{\rm z}(z_i-z_{i-1})+ \zeta'(z_i)\left( \frac{\mu}{2} |(\bFe)_i|^2 + \Gamma(\det(\bFe)_i) \right) \ni 0.\label{xi0d}  
\end{align}
In particular, time derivatives in \eqref{stress0}-\eqref{xi0} are approximated by incremental quotients. All nonlinearities are evaluated implicitly, with the exception of the state dependencies in the dissipation potential $R(\bP,z,\dot \bP,\dot z)$. This is paramount to obtain the specific form of relations \eqref{P0d}-\eqref{xi0d}. 

The minimality in \eqref{inc_min_0D} yields $z_i\geq 0$ by induction. Let $z_{i-1}\geq 0$ and assume by contradiction that $z_i<0$. As $\zeta$ and $\rho$ are monotone in $z$ and constant on $(-\infty,0]$, so are the energy and the plastic dissipation. On the other hand, $R_{\rm z}(z_i - z_{i-1}) > R_{\rm z}(0-z_{i-1})$. This proves that $z_i <0$ cannot be a minimizer in \eqref{inc_min_0D}. Existence of minimizers of \eqref{inc_min_0D} is trivial since all terms involved in the minimization problem are lower semicontinuous and the state space $\RR^{d\times d} \times \SL \times [0,z_{i-1}]$ is closed.

\subsection{Assessment of model performance}\label{Subsec:ParameterAnalysis}
We illustrate now the performance of the model and the role of its various parameters. The incremental strategy of Subsection~\ref{sec:time-discrete-frame} is implemented by prescribing the stress history $t \mapsto \bsigma(t)$ and the initial conditions $\bP_0=\bI$ and $z_0 =1$ (undamaged) and solving (a suitably regularized version of) \eqref{inc_min_0D} via the Newton Method (see Subsection~\ref{Subsec:ngsolve}). 

The material  parameters are given in Table~\ref{tab:material_parameters} and are fixed throughout the rest of the paper if not specified otherwise. Parameters $E$, $\nu$, $H$, and $\sigma_{\rm p}$ relate to steel-like materials, see, e.g. \cite{ASTM19}. The Lam\'e parameters $\lambda$ and $\mu$ are computed via
$$\lambda =\frac{E\nu}{(1+\nu)(1-2\nu)},\quad\quad\, \mu =\frac{E}{2(1+\nu)}.$$  
\begin{table}[h!]
	\begin{tabular}{cccccccc}
		$E$ [GPa] & $\nu$ & $A$ & $\sigma_{\rm p}$ [MPa] & $H$ [MPa] & $\sigma_{\rm z}$ [MPa] & $\rho_0$ & $\zeta_0$ \\
		\hline 
		210 & 0.3 & 450 & 250 & 650 & $0.4$ & 0.5 & 0.5
	\end{tabular}
	\caption{Young's modulus, Poisson's ratio, maximal stress, plastic yield stress, hardening parameter, damage yield stress, and coupling parameters.}
	\label{tab:material_parameters}
\end{table}

\begin{figure}[h!]
	\centering
	\includegraphics[width=0.45\textwidth]{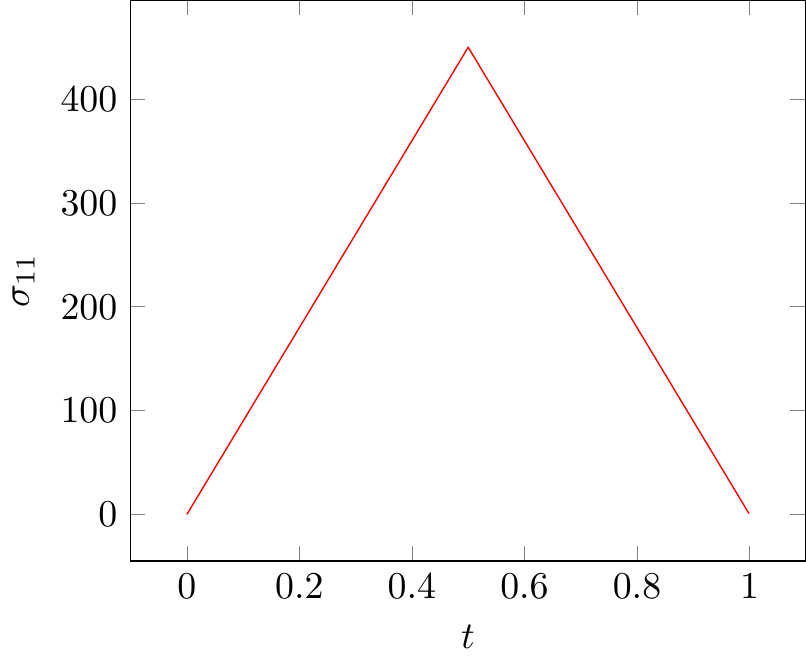}
	\includegraphics[width=0.45\textwidth]{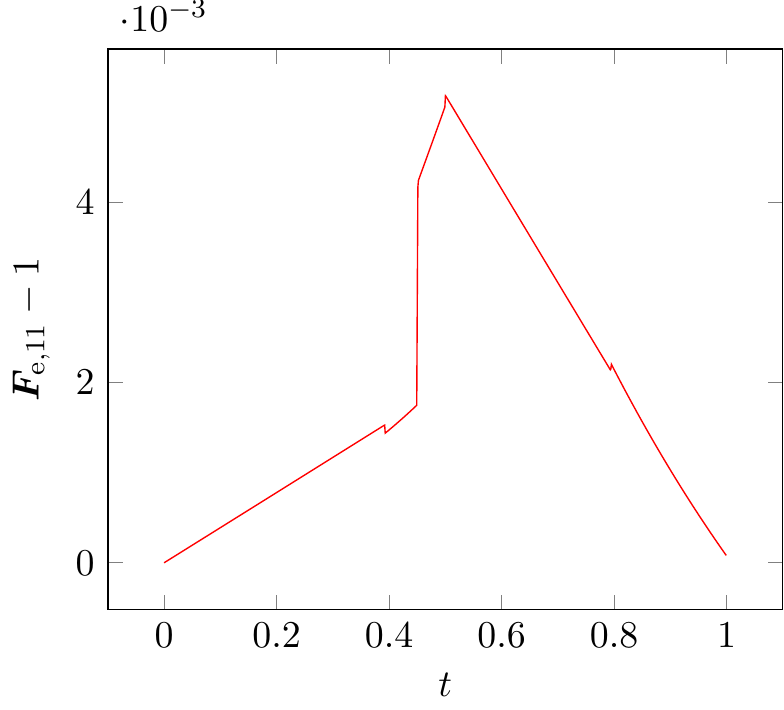}
	
	\includegraphics[width=0.45\textwidth]{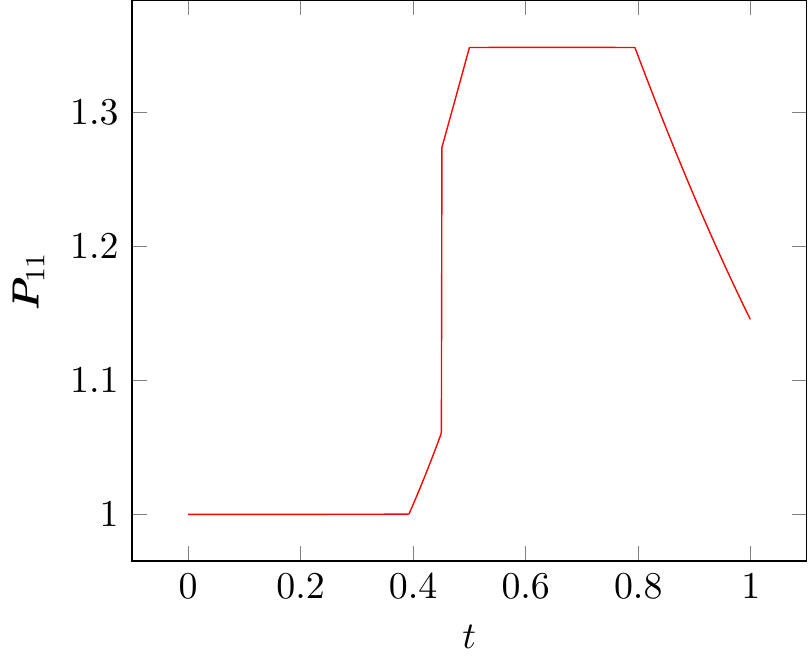}
	\includegraphics[width=0.45\textwidth]{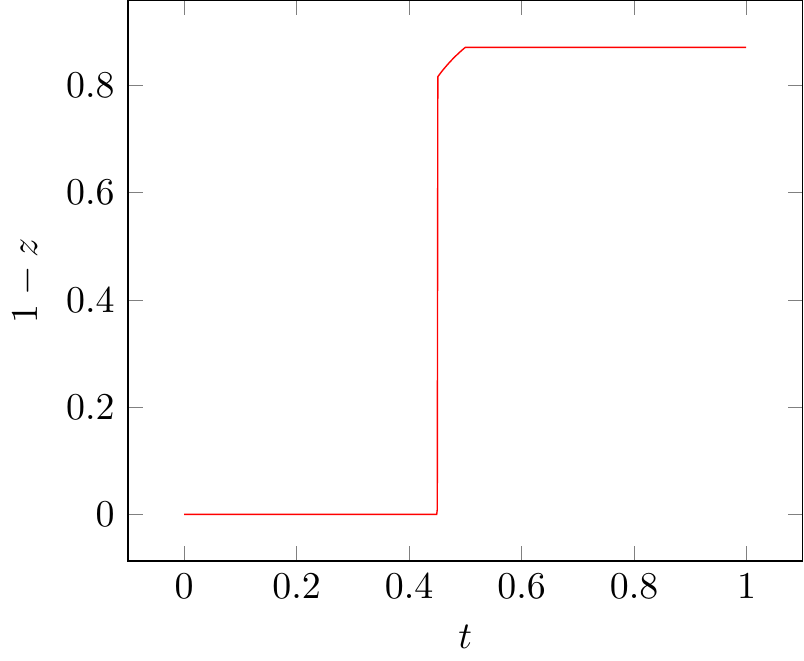}

	\caption{Evolution of stress (top left), elastic strain (top right), plastic strain (bottom left), and damage (bottom right) for the base parameters in Table~\ref{tab:material_parameters}.}
	\label{fig:res_0D_sigmaxx}
\end{figure}
\begin{figure}[h!]
	\centering
	\includegraphics[width=0.45\textwidth]{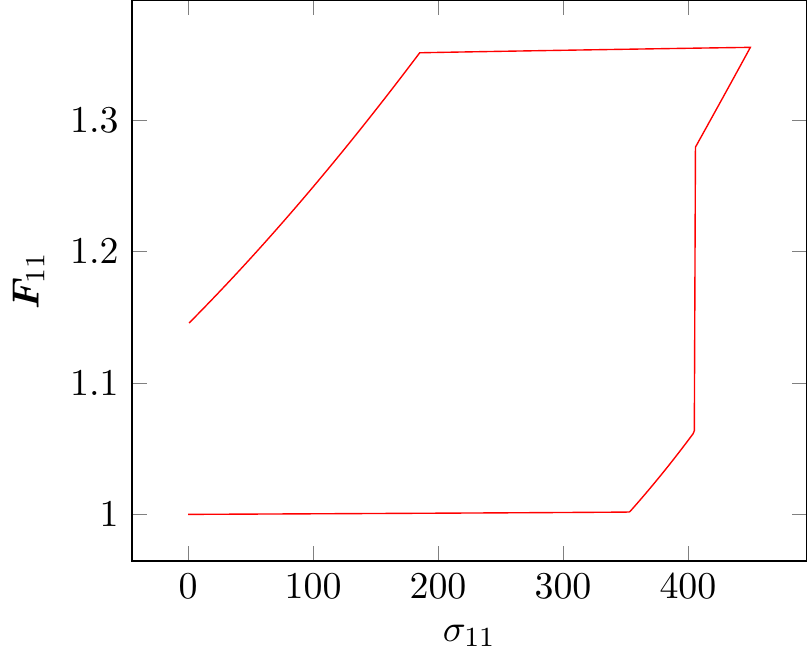}
	\includegraphics[width=0.45\textwidth]{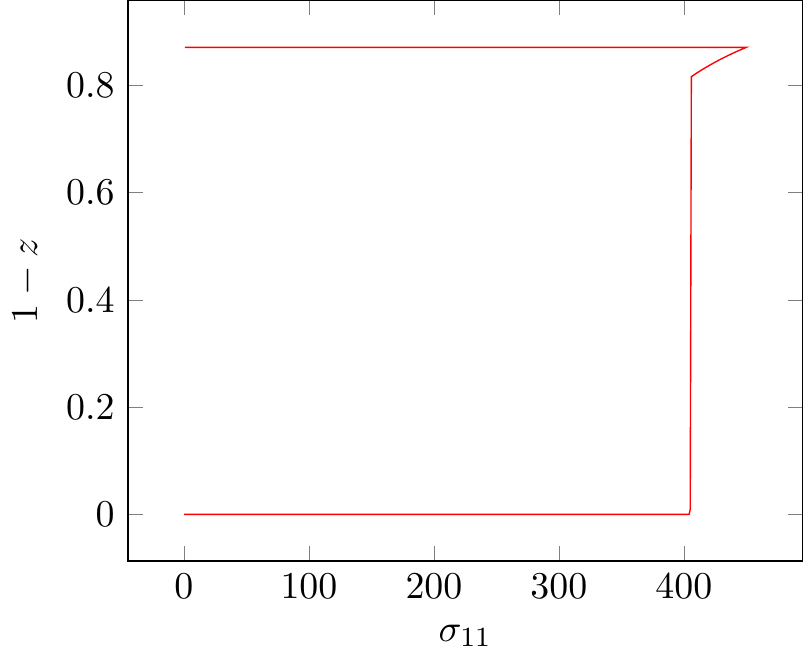}
	
	\caption{Stress-strain (left) and stress-damage curves (right) for the base parameters in Table~\ref{tab:material_parameters}.}
	\label{fig:res_0D_sigmaxx_1}
\end{figure}

We start by an uniaxial test and prescribe
\begin{equation*}
\bsigma(t) = \begin{pmatrix}
\sigma_{11}(t) & 0  \\
0 & 0   
\end{pmatrix}, \quad t\in [0,1],
\end{equation*}
where $\sigma_{11}(t)=A\arcsin(\sin(\pi t))/(\pi/2)$ and $A>0$ denotes the maximal stress. We choose the time step size to be $\tau = 10^{-4}$ for all simulations in this section. 

Under loading, the material behaves elastically  up to time $t=0.4$. Then, plasticization occurs without damage up to time $t=0.45$. At this point, the material damages and plasticizes simultaneously, first abruptly and then more regularly. Note that this fast plasticization is the effect of damage. Indeed, upon damaging the effective plastic yield stress \eqref{effe} drops and plasticization rapidly develops. By removing the load, damage is unrecovered and the material plasticizes again, for the effective plastic yield stress \eqref{effe} has lowered. 

Figures~\ref{fig:res_0D_sigmaxx} and \ref{fig:res_0D_sigmaxx_1} show the results for the choice of parameters in Table~\ref{tab:material_parameters}. In the following, we analytically discuss the qualitative behavior induced by changing these parameters.

\textit{Plastic yield stress.}
For smaller values of the plastic yield $\sigma_{\rm p}$, the plastic deformation is activated at smaller stresses. This behavior is observed both at loading and at unloading phase, see Figure~\ref{fig:res_0D_sigma_p}. As a consequence, damage gets activated earlier or later according to plasticity, since the thermodynamic damage force $\xi$ depends on the plastic strain through the elastic energy density $\widehat \We$, see equation \eqref{xi}.

\begin{figure}[h!]
	\centering
	\includegraphics[width=0.45\textwidth]{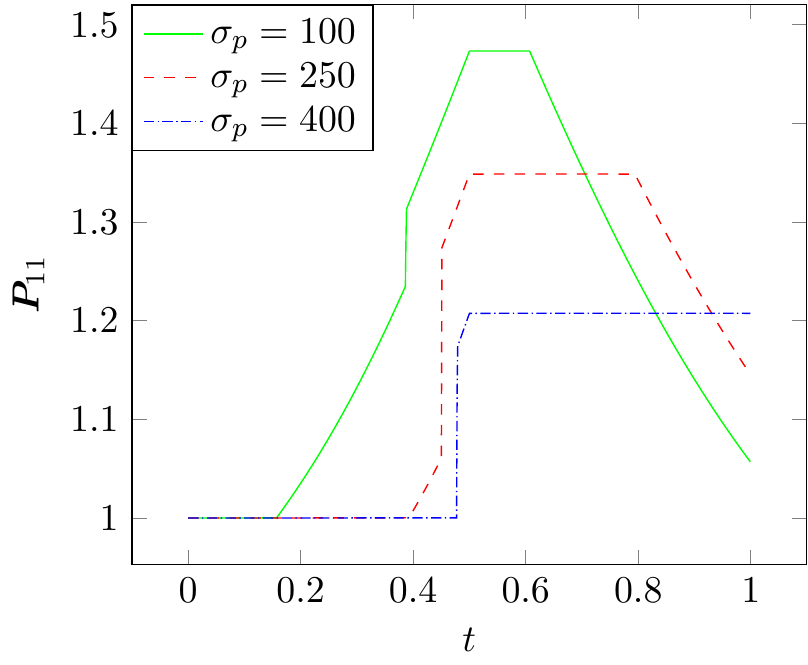}
	\includegraphics[width=0.45\textwidth]{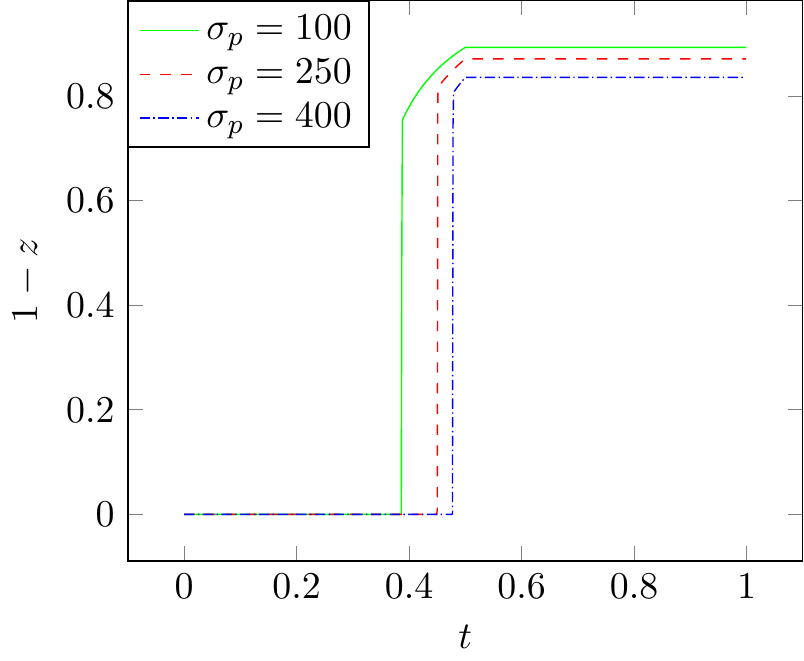}
	\caption{Effect of changing the plastic yield stress $\sigma_{\rm p}$. Plastic strain (left) and damage (right) for different choices of the parameter.}
	\label{fig:res_0D_sigma_p}
\end{figure}

\textit{Hardening.}
Variations of the kinematic hardening parameter $H$ lead to different slopes in the stress-strain curves once plastic deformation is activated. This behavior can be seen in Figure~\ref{fig:res_0D_hardening}. We remark that in all three examples plasticity is activated at the same stress level. Comparing the plastic strain for different hardening parameters, we see that these influence the maximal plastic strain. The damage curves for different $H$ are almost identical.

\begin{figure}[h!]
	\centering
	\includegraphics[width=0.45\textwidth]{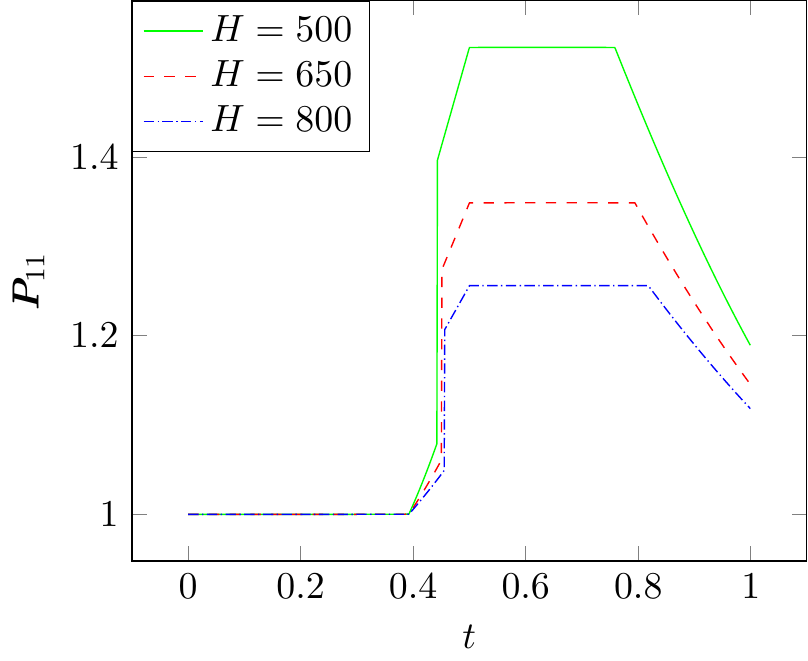}
	\includegraphics[width=0.45\textwidth]{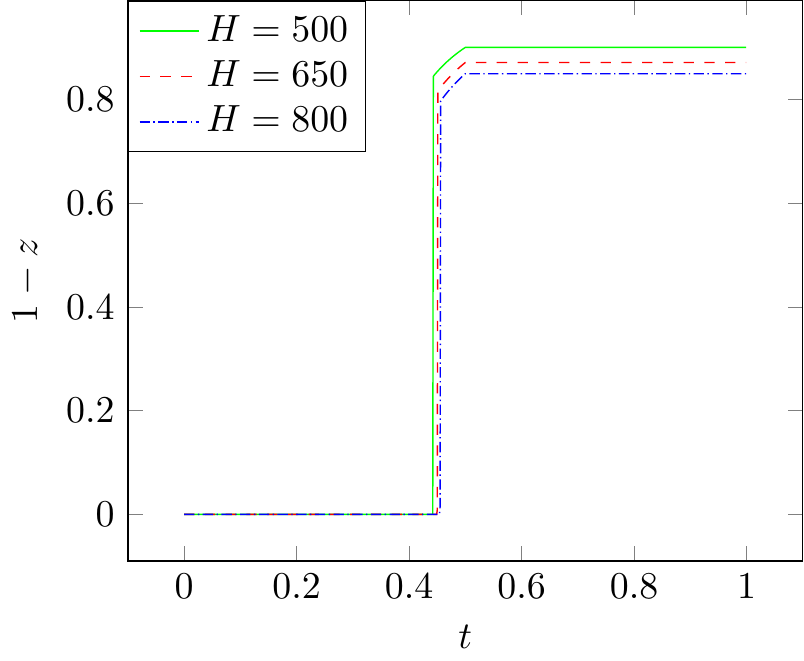}
	\caption{Effect of changing the hardening modulus $H$. Plastic strain (left) and damage (right) for different choices of the parameter.}
	\label{fig:res_0D_hardening}
\end{figure}

\textit{Damage yield stress.}
Smaller damage-yield stresses $\sigma_{\rm z}$ induce earlier material damage compared to our base example. Remarkably, the final plastic strain at time $t=1$ is almost the same for all three examples. Moreover, for a material that is fully damaged before plasticity is activated ($\sigma_z=0.01$), we observe a linear plasticization, see Figure~\ref{fig:res_0D_sigma_z}.
\begin{figure}[h!]
	\centering
	\includegraphics[width=0.45\textwidth]{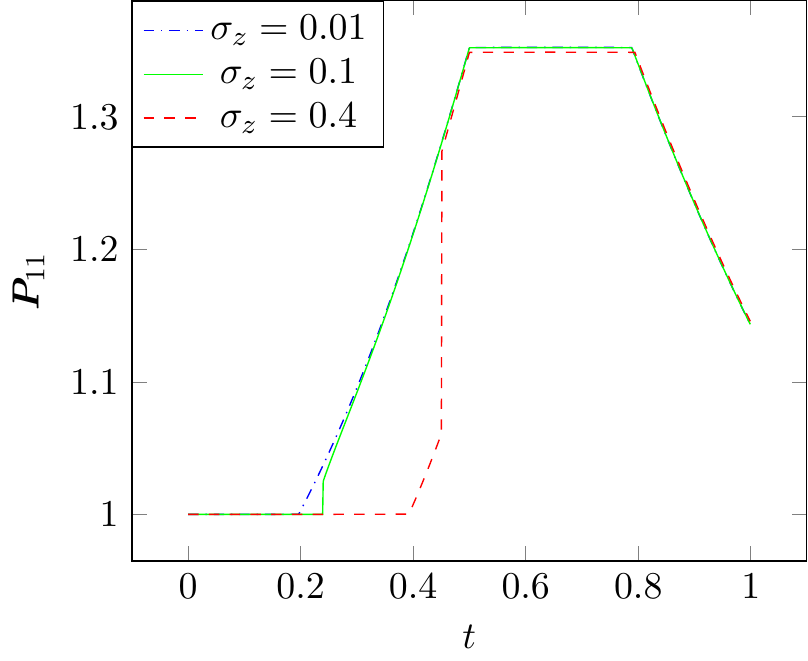}
	\includegraphics[width=0.45\textwidth]{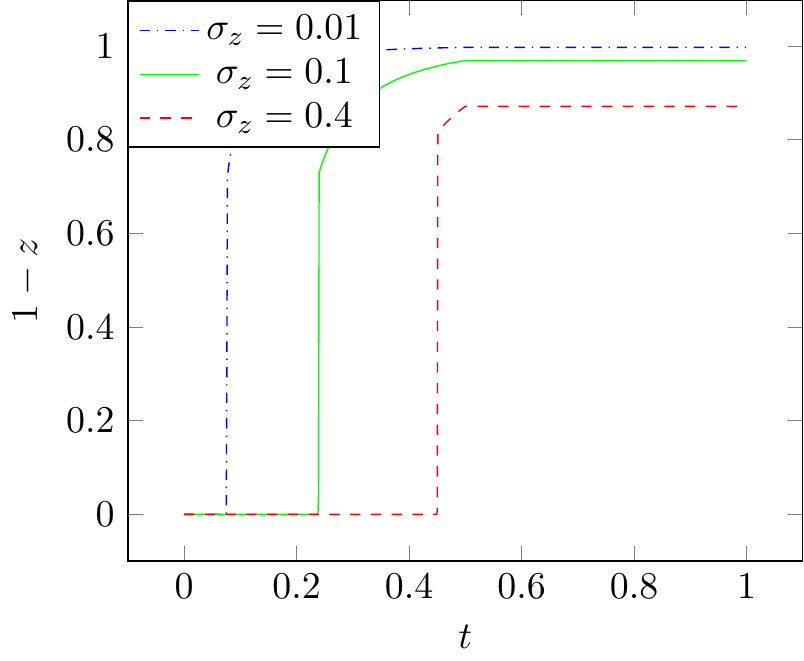}
	\caption{Effect of changing the damage yield stress $\sigma_{\rm z}$. Plastic strain (left) and damage (right) for different choices of the parameter.}
	\label{fig:res_0D_sigma_z}
\end{figure}

\textit{Dissipation-coupling parameter.}
For the smaller coupling parameter $\rho_0=0.1$, we observe a larger jump in the plastic strain, see Figure~\ref{fig:res_0D_rho}. This is due to fact that the smaller $\rho_0$ the easier it is for the material to deform plastically after it is damaged. We observe that without the coupling in the dissipation ($\rho_0=1$), no jump in the plastic-strain evolution occurs.
\begin{figure}[h!]
	\centering
	\includegraphics[width=0.45\textwidth]{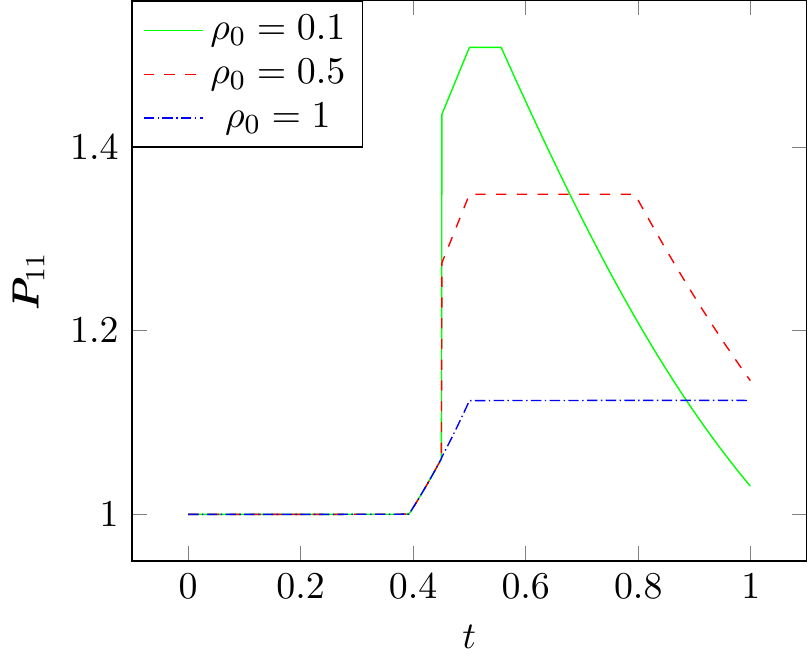}
	\includegraphics[width=0.45\textwidth]{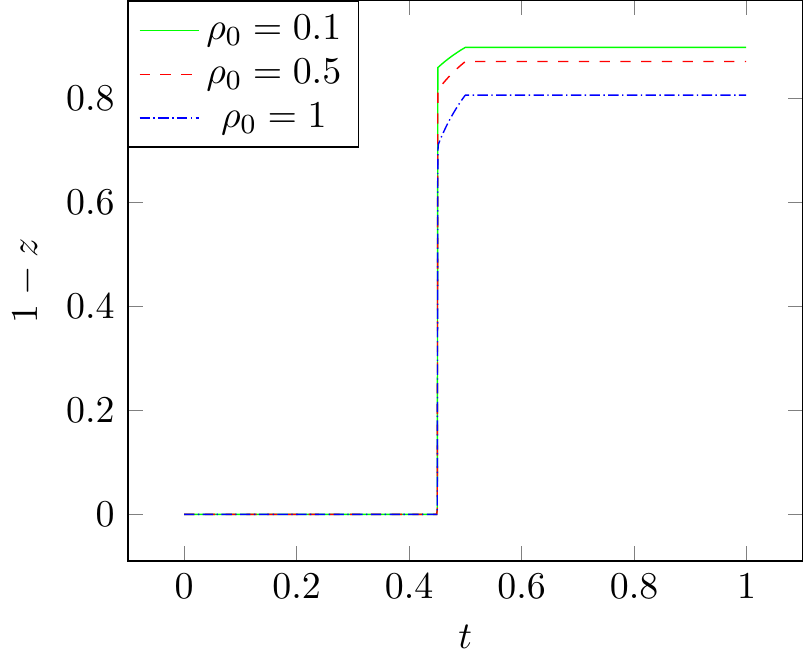}
	\caption{Effect of changing the dissipation-coupling parameter $\rho_0$. Plastic strain (left) and damage (right) for different choices of the parameter.}
	\label{fig:res_0D_rho}
\end{figure}

\textit{Energy-degradation parameter.}
The parameter $\zeta_0$ influences the activation onset of damage via the thermodynamic driving force $\xi$ in equation \eqref{xi}. Moreover, $\zeta$ changes the elastic properties making the material softer upon damage. Consequently, for smaller values of $\zeta_0$ the material reaches a larger elastic deformation, see Figure~\ref{fig:res_0D_zeta}.
\begin{figure}[h!]
	\centering
	\includegraphics[width=0.45\textwidth]{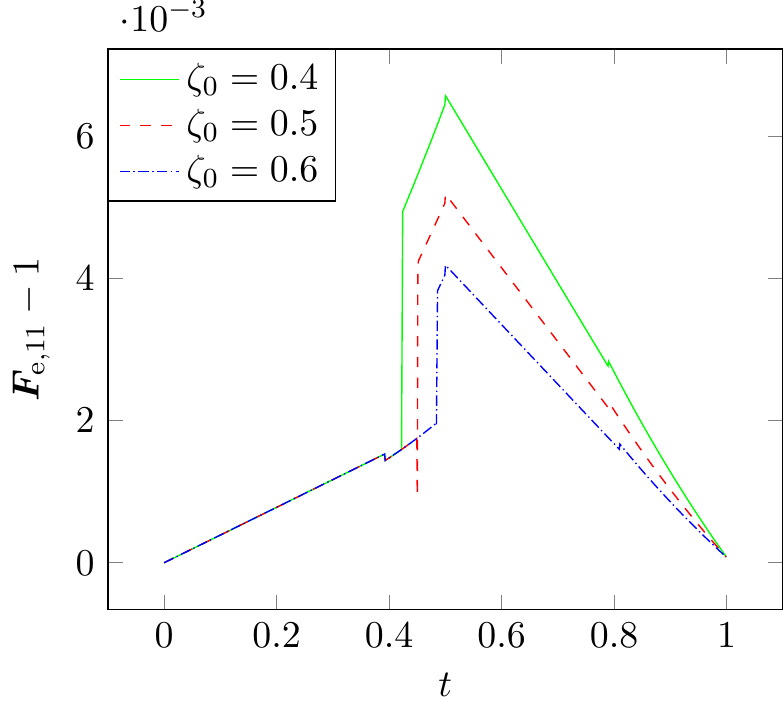}
	\includegraphics[width=0.45\textwidth]{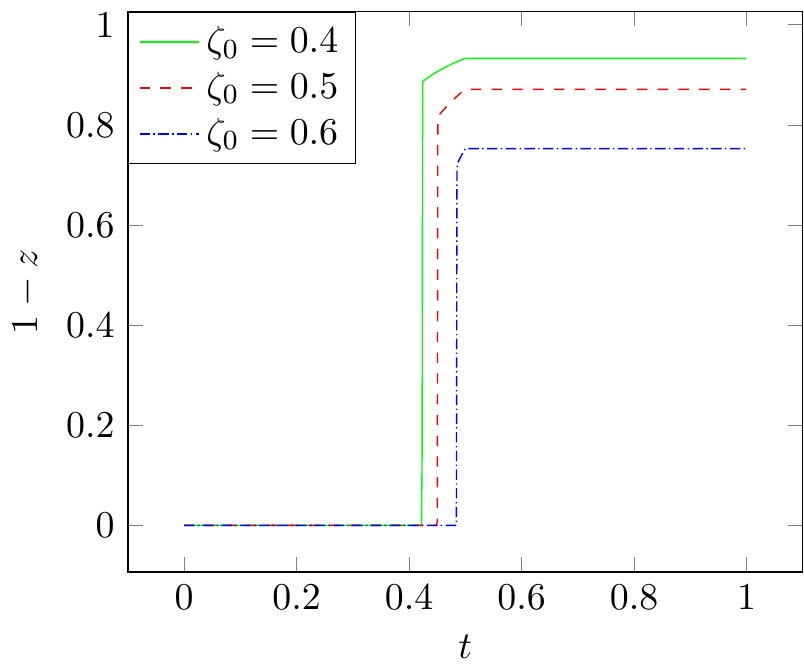}
	\caption{Effect of changing the energy-degradation parameter $\zeta_0$. Elastic strain (left) and damage (right) for different choices of the parameter.}
	\label{fig:res_0D_zeta}
\end{figure}

We close this section by showing a biaxial non-proportional square-shaped test in a plane-strain setting. The evolution of the elastic strain along a square-shaped loading history 
$$\bsigma(t) = \left(
\begin{matrix}
\sigma_{11}(t)&0  \\ 0& \sigma_{22}(t) 
\end{matrix}
\right)
$$
is reported in Figure~\ref{mp_0D_biaxial_square_stress}. The experiment is performed under the choice of parameters from Table~\ref{tab:material_parameters} with the exception of $\rho_0=\zeta_0=0$ (no damage) and with $E=4$ GPa.

\begin{figure}[h!]
	\centering
	\includegraphics[width=0.45\textwidth]{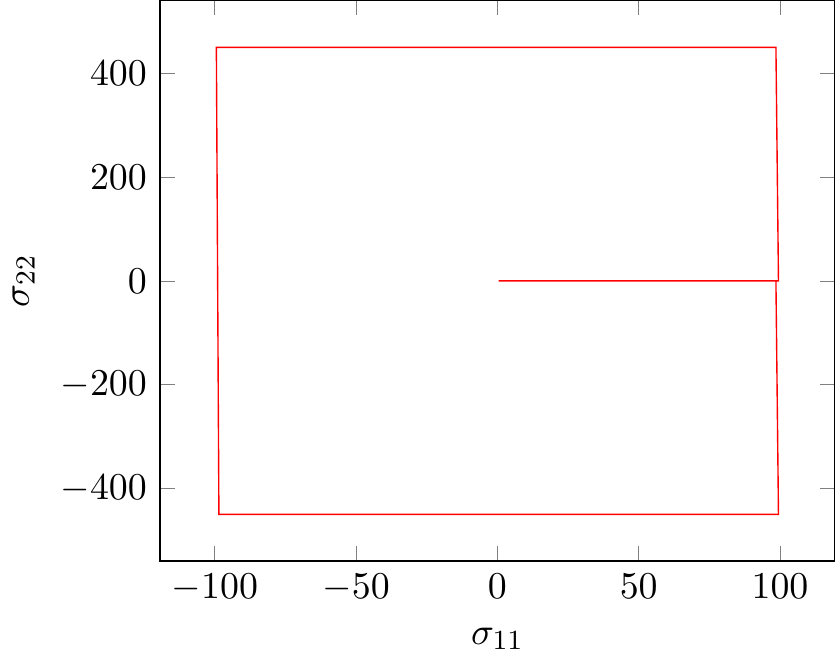}
	\includegraphics[width=0.45\textwidth]{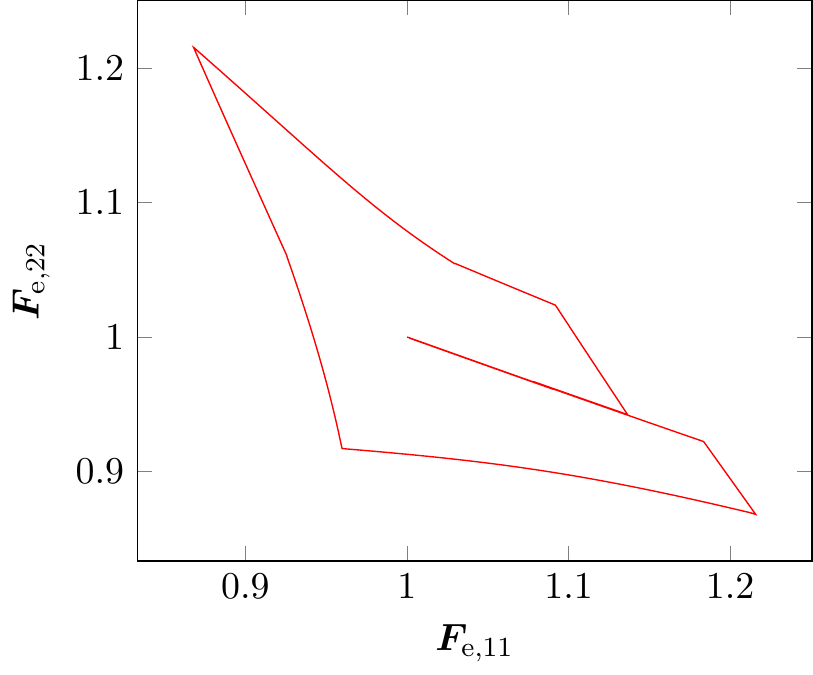}
	\caption{Stress history (left) and elastic strain (right) in a biaxial test.}
	\label{mp_0D_biaxial_square_stress}
\end{figure}

\subsection{Algorithmical aspects}\label{Subsec:ngsolve}
Newton's method calls for differentiability of all terms in the incremental minimization problem. Therefore, we regularize all non-differentiable terms appearing in \eqref{inc_min_0D}, as well as in \eqref{inc_min_3D} below. More precisely, we let $\varepsilon >0$ and replace $R_{\rm z}$ by the approximation
\begin{equation}\label{def:Psiepsilon}
R_{\rm z}^\varepsilon(\dot z) = \sigma_{\rm z}  \begin{cases}
-\dot z, &\text{if }\dot z < -\varepsilon,\\
-\dot z + \displaystyle \frac{(\dot z + \varepsilon)^3}{3\varepsilon^2}, &\text{else.}
\end{cases}
\end{equation}
satisfying $R_{\rm z}^\varepsilon \in C^2(\RR), (R_{\rm z}^\varepsilon)'(0)=0$. A simple calculation reveals
\begin{equation*}
\partial R_{\rm z}^\varepsilon(\dot z) = \begin{cases}
-\sigma_{\rm z}, &\text{if }\dot z < -\varepsilon,\\
-\sigma_{\rm z} + \displaystyle\frac{(\dot z + \varepsilon)^2}{\varepsilon^2}, &\text{else.}
\end{cases}
\end{equation*}
Secondly, we replace $\widehat R_{\rm p}$ by
$$\widehat R_{\rm p}^\varepsilon(\bA) = \sigma_{\rm p}\left(\sqrt{\bA:\bA + \varepsilon^2}-\varepsilon\right).$$
This leads to 
\begin{equation*}
\partial \widehat R_{\rm p}^\varepsilon(\bA) = \sigma_{\rm p} \frac{\bA}{\sqrt{\bA:\bA+\varepsilon^2}}.
\end{equation*}

In order to choose a proper regularization parameter, we study results for the base example in Subsection~\ref{Subsec:ParameterAnalysis} (Table~\ref{tab:material_parameters}) with different choices of $\varepsilon$. In Figure~\ref{fig:res_0D_eps} one can see the stress-strain curves and elastic strain curves for the values of $\varepsilon=10^{-4},10^{-5},10^{-6},10^{-7}$. We observe that for larger $\varepsilon$ the qualitative and quantitative behavior highly differs from the one with smaller $\varepsilon$ since the regularization pollutes the real solution. For  $\varepsilon\leq 10^{-7}$ trajectories are almost coinciding. As Newton's minimization gets more unstable for smaller $\varepsilon$, we fix $\varepsilon=10^{-7}$ for all our computations in Subsection~\ref{Subsec:ParameterAnalysis}, as well as in Section~\ref{Sec:Examples} below. 
\begin{figure}[h!]
	\centering
	\includegraphics[width=0.45\textwidth]{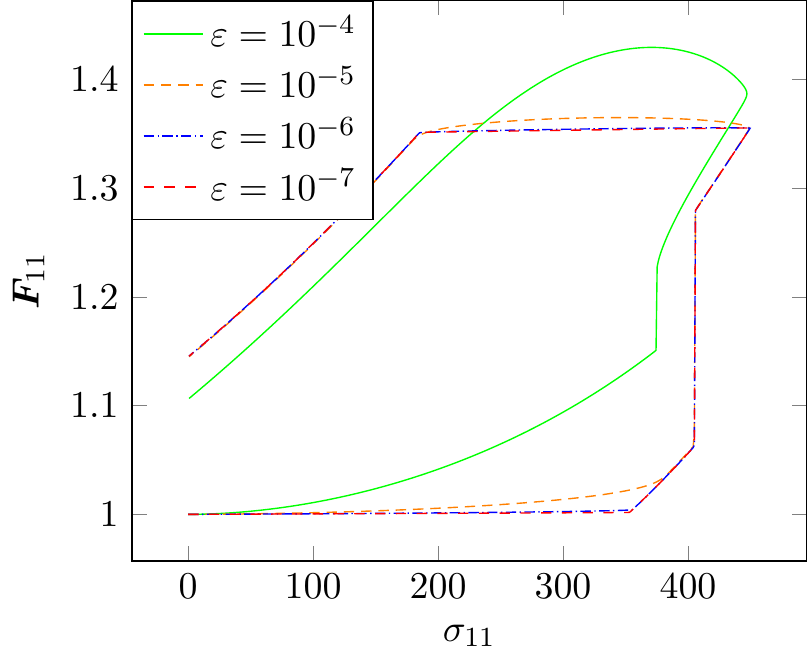}
	\includegraphics[width=0.45\textwidth]{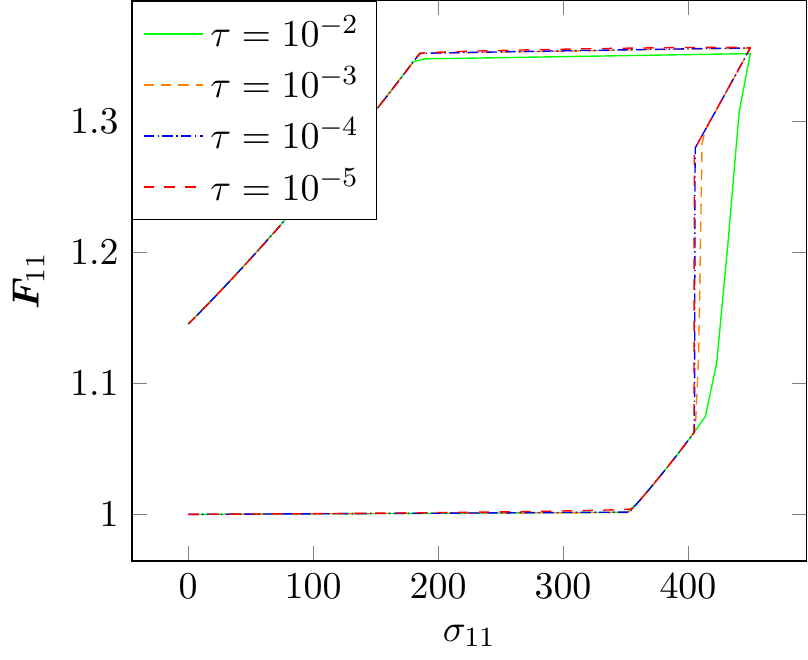}
	\caption{Stress-strain curves for $\tau=10^{-4}$ and different values of $\varepsilon$ (left) and $\varepsilon=10^{-7}$ and different values of $\tau$ (right) computed on the base example in Subsection~\ref{Subsec:ParameterAnalysis}.}
	\label{fig:res_0D_eps}
\end{figure}

In Figure~\ref{fig:res_0D_eps} we display results for the same base example (with $\varepsilon=10^{-7}$) for different values of the time step size, namely $\tau=10^{-2}, 10^{-3}, 10^{-4}, 10^{-5}$. We observe very inaccurate results for $\tau=10^{-2}$ which improve fast for smaller values of $\tau$. It seems that for values $\tau \le 10^{-4}$ the solutions are for all purposes indistinguishable. However, computation time grows linearly in $1/\tau$. For this reason, we choose $\tau=10^{-4}$ whenever possible. For some of the examples in Section~\ref{Sec:Examples} it is however necessary to choose an even smaller time step to guarantee convergence of Newton's algorithm.

\section{Quasistatic evolution}\label{sec:quasi}

We now combine the constitutive model \eqref{stress0}-\eqref{xi0} with the quasistatic equilibrium system and with boundary conditions. To this aim, we introduce the {\it reference configuration} of the body as $\Om\subset \RR^d $, being a bounded, open, connected set with a Lipschitz boundary $\partial \Om$. We denote by $\bx \in \Om$ the generic point in the reference configuration.

\subsection{Equilibrium system}
The {\it deformation} of the body is described by a map $\bvarphi: [0,T]\times \Om \to \RR^d$. We assume $\partial \Om$ to be split into relatively open {\it Dirichlet} and {\it Neumann} parts, denoted by $\Gamma_D$ and $\Gamma_N$, respectively, such that $\partial \Om = \overline{\Gamma_D} \cup \overline{\Gamma_N}$ (closure in $\partial \Om$), where $\Gamma_D$ has positive surface measure. On $\Gamma_D$ we prescribe the boundary condition $$\bvarphi = \bphi_{\rm Dir} \quad \text{on }[0,T] \times \Gamma_D,$$ where $\bphi_{\rm Dir}:[0,T]\times \Gamma_D \to \RR^d$ is given and represents a possibly time-dependent, prescribed boundary deformation.
On $\Gamma_N$, a surface-force density (traction) $\bg:[0,T]\times \Gamma_N \to \RR^d$ is prescribed in form of a Neumann boundary condition, see \eqref{Neumann}. Moreover, the body may be subjected to a bulk-force density $\bforce: [0,T]\times \Om \to \RR^d$. 

Taking into account the body force $\bforce$ and the traction force $\bg$, the 1st Piola-Kirchhoff stress tensor $\bsigma$ in \eqref{stress} is required to satisfy the {\it quasistatic momentum balance}
\begin{align}
-\nabla \cdot \bsigma &= \bforce \quad \text{in }(0,T) \times \Om,\\
\bsigma \cdot \bn &= \bg \quad \text{in }(0,T) \times \Gamma_N, \label{Neumann}
\end{align}
where $\bn = \bn(\bx)$ denotes the outward normal to $\partial \Om$, which is almost everywhere well-defined.

\subsection{Regularization}
In the quasistatic setting, the problem can be regularized by augmenting the energy by second-length-scale terms featuring gradients of the internal variables. This would guarantee strong compactness of the internal states $\bP$ and $z$ and would allow to show existence of suitable weak solutions, see \cite{MSZ19}. In order to move in this direction, one has to replace $\bX$ and $\xi$ by 
\begin{align}
\tilde \bX &= \bX - \mup \Delta \bP, \\
\tilde \xi &= \xi - \muz \Delta z,
\end{align}
where $\Delta$ denotes the Laplace operator and the scale parameters $\mup,\muz \geq 0$ are given.

\subsection{Incremental minimization}
Following the blueprint of Subsection~\ref{sec:time-discrete-frame}, we address a uniform discretization of the time interval $[0,T]$ with step-size $\tau>0$ (non-uniform time steps can be considered as well). At every time step the {\it quasistatic} incremental minimization problem reads as follows:

Given $(\bP_{i-1},z_{i-1}): \Om \to \SL \times [0,1]$ and $\boldsymbol{f}_i := \boldsymbol{f}(t_i), \ \boldsymbol{g}_i := \boldsymbol{g}(t_i)$, find
\begin{align}\label{inc_min_3D}
(\bvarphi_i,\bP_i,z_i) \in \argmin &\bigg\{ \int_\Om
  \zeta(z(\bx))\widehat \We(\bnabla \bvarphi(\bx)(\bP(\bx))^{-1}) +
  \Wp(\bP(\bx))\d \bx \nonumber\\
&
+ \int_\Omega \frac{\muz}{2}|\nabla z(\bx)|^{2} + \frac{\mup}{2}|\nabla \bP(\bx)|^2 \d \bx \nonumber \\ 
&- \int_{\Om} \boldsymbol{f}_i(\bx)\cdot \bvarphi(\bx) \d \bx - \int_{\Gamma_N} \bg_i(\bx) \cdot \bvarphi(\bx) \d S(\bx) \nonumber \\
&+ \int_{\Om} R(z(\bx)-z_{i-1}(\bx)) \d \bx \nonumber\\
&+\int_{\Om} \rho(z_{i-1}(\bx))
  R_{\rm p}(\bP_{i-1}(\bx), \bP(\bx)- \bP_{i-1}(\bx)) \d \bx \bigg\}.
\end{align}
More precisely, as the minimization involves gradient terms, the minimum is properly taken in the space of Sobolev functions $\bvarphi \in H^1(\Om;\RR^d)$ satisfying $\bvarphi = \bphi_{\rm Dir}$ on $\Gamma_D$, $\bP \in H^1(\Om;\RR^{d\times d})$ with $\bP \in \SL$ almost everywhere, and $z \in H^1(\Om;\RR)$ with $z \in [0,1]$ almost everywhere.

\subsection{Existence}\label{Subsec:existence}
 We present now an existence result for  problem
\eqref{inc_min_3D}. The proof is based on the  Direct Method of the
Calculus of Variations, hinging upon weak lower
semicontinuity, coercivity of the driving functional, and closure of
the state space. In Theorem
\ref{thm:exist_2d}, we present a statement in two dimensions, under the exact
modeling  assumptions introduced above. In three space
dimensions, existence follows by slightly strengthening the
assumptions on the elastic response,  see Remark \ref{rem:exist_3d}.

\begin{theorem}[Existence for $d=2$]\label{thm:exist_2d}
Let $(\bP_{\rm old},z_{\rm old}) \in H^1(\Om; \RR^{2\times 2} \times
\RR)$ with $\bP_{\rm old} \in \mathrm{SL}(2)$ and $z_{\rm old} \in
[0,1]$ almost everywhere. Further, let $\bforce \in
L^{ 2}(\Om;\RR^2), \bg \in L^{ 2}(\Gamma_N;\RR^2)$ and $\mup>0, \muz \ge 0$. Assume that there exists $\widehat \bvarphi_{\rm Dir} \in H^1(\RR^2;\RR^2)$ such that $\widehat \bvarphi_{\rm Dir}=\bvarphi_{\rm Dir}$ on $\Gamma_D$. 
We define the functional 
\begin{align}\label{functional}
\mathcal F(\bvarphi,\bP,z) &= \int_\Om
  \zeta(z(\bx))\widehat \We(\bnabla \bvarphi(\bx)(\bP(\bx))^{-1}) +
  \Wp(\bP(\bx))\d \bx \nonumber\\
&
+ \int_\Omega \frac{\muz}{2}|\nabla z(\bx)|^{2} + \frac{\mup}{2}|\nabla \bP(\bx)|^2 \d \bx \nonumber \\ 
&- \int_{\Om} \boldsymbol{f}(\bx)\cdot \bvarphi(\bx) \d \bx - \int_{\Gamma_N} \bg(\bx) \cdot \bvarphi(\bx) \d S(\bx) \nonumber \\
&+ \int_{\Om} R(\bP_{\rm old}(\bx),z_{\rm old}(\bx),\bP(\bx)-\bP_{\rm old}(\bx),z(\bx)-z_{\rm old}(\bx))\d \bx
\end{align}
and the state space
$$\mathcal A = \bigg\{ (\bvarphi,\bP,z) \in H^1(\Om;\RR^2 \times \RR^{2\times 2} \times \RR) : \bvarphi = \bvarphi_{\rm Dir} \text{ on } \Gamma_D, \bP \in \mathrm{SL}(2), z \in [0,1] \text{ a.e.} \bigg\}.$$
Then $\mathcal F$ attains a minimizer on $\mathcal A$.
\end{theorem}

\begin{proof}  Within this proof, we use the symbol $C$ for a 
  generic positive constant depending on data, possibly changing from
   line to line.

 Let us start by checking the coercivity of $\mathcal F$. We  use the Poincar\'e inequality and the fact that $\bforce \in
L^{ 2}(\Om;\RR^2)$ and $ \bg \in L^{ 2}(\Gamma_N;\RR^2)$ in order to
estimate the loading  terms as
\begin{align}
 \left|\int_{\Om} \boldsymbol{f}(\bx)\cdot \bvarphi(\bx) \d \bx + \int_{\Gamma_N} \bg(\bx) \cdot \bvarphi(\bx) \d S(\bx)\right|
 \le C\|\bnabla\bvarphi\|_{L^{ 2} }. \label{est:1}
\end{align}
 We can estimate the above right-hand side by Young's inequality  and the uniform bound
\eqref{conscomp} as   
\begin{equation}\label{est:2}
\|\bnabla\bvarphi\|_{L^{ 2} } \le  \frac{ \zeta_0\mu}{2}  \|\bnabla
\bvarphi \bP^{-1}\|_{L^2 }^2 + \frac{1}{ 2 \zeta_0\mu }
\|\bP\|_{L^{ \infty}}^2  \le \frac{ \zeta_0\mu}{2}  \|\bnabla
\bvarphi \bP^{-1}\|_{L^2}^2 + \frac{ c_K^2}{
  2 \zeta_0\mu } ,
\end{equation}
 where $\mu$ is the Lam\'e parameter. 
As $ \widehat{\We}(\bFe) \ge \mu  |\bFe|^2/2  - C$,  by  the uniform bound \eqref{conscomp},  estimates \eqref{est:1}-\eqref{est:2}, and $\zeta \ge \zeta_0>0$, we deduce
\begin{align*}
\mathcal F(\varphi,\bP,z) & \ge  \frac{\zeta_0\mu}{2} 
                            \|\bnabla\bvarphi \bP^{-1}\|_{L^2}^2 + \mup
                            \|\bnabla\bP\|_{L^2}^2 + \muz \|\nabla
                            z\|_{L^2}^2 -  C \|\bnabla\bvarphi
                            \|_{L^2}   - C \\
& \ge \|\bnabla\bvarphi
                            \|_{L^2}^2
                        -
                            \frac{c_K^2}{2 \zeta_0\mu}  + \mup
                            \|\bnabla\bP\|_{L^2}^2 + \muz \|\nabla
                            z\|_{L^2}^2 -  C  \|\bnabla\bvarphi
                            \|_{L^2}  - C \\
& \ge  \frac12\|\bnabla\bvarphi
                            \|_{L^2}^2
+ \mup
                            \|\bnabla\bP\|_{L^2}^2 + \muz \|\nabla
                            z\|_{L^2}^2 -  C.
\end{align*}
 We hence have that 
\begin{equation}
\|\bnabla\bvarphi\|_{L^2 }^2 +
\|\bP\|_{L^\infty }^2+    
\|z\|_{L^\infty }^2+   \mup
\|\bnabla\bP\|_{L^2 }^2 + \muz \|\nabla
z\|_{L^2 }^2 \le  C(\mathcal F(\bvarphi,\bP,z)+1),  \label{coe}
\end{equation}
where  the uniform bound on $\bP$ comes from \eqref{conscomp}
and that  on $z$  follows as  $(\bvarphi,\bP,z) \in \mathcal A$ implies $z \in [0,1]$.

 Note that, under the assumptions on $\bvarphi_{\rm Dir}$, the set
$\mathcal A$ of admissible states is nonempty.  Thus, we may take an infimizing sequence $(\bvarphi_n,\bP_n,z_n) \in \mathcal A$, namely,
$$\lim_{n\to \infty} \mathcal F(\bvarphi_n,\bP_n,z_n) = \inf_\mathcal A \mathcal F > -\infty.$$
In particular, $\mathcal F(\bvarphi_n,\bP_n,z_n) $  is bounded,
independently of $n$. 
By the coercivity property \eqref{coe}, we can take a (not-relabeled)
subsequence such that  $(\bvarphi_n,\bP_n,z_n) \to
(\bvarphi,\bP,z)$ weakly in $H^1$.  As the sequence $\bvarphi_n$
satisfies the determinant constraint $\det \bnabla\bvarphi_n > 0$, we
have that $\det\bnabla \bvarphi_n \to \det\bnabla \bvarphi$ weakly in
$L^1(\Om;\RR)$ \cite[Prop. 3.2.4]{Kruzik-Roubicek}.  In particular
$\bvarphi=\bvarphi_{\rm Dir}$ on $\Gamma_D$.
At the same
time,   $\det \bP_n \to \det \bP$  strongly  in $L^1(\Om)$,  implying
$\det \bP=1$. Further,  as $z_n \in [0,1]$ almost everywhere, we
have that $z \in [0,1]$ almost everywhere as well. Hence
$(\bvarphi,\bP,z) \in \mathcal A$,  which is indeed closed in the
weak topology of $H^1$. 

The above convergences entail in particular that $\bnabla \bvarphi_n
\bP^{-1}_n \to \bnabla \bvarphi
\bP^{-1}$ weakly in $W^{1,q}$ for all $q<2$ and we have that $\det (\bnabla \bvarphi_n
\bP^{-1}_n) = \det \bnabla \bvarphi_n \to \det \bnabla \bvarphi = \det (\bnabla \bvarphi
\bP^{-1})$ weakly in $L^1$. Observe now that the elastic energy is polyconvex \cite{Ball76},
namely, we can write $\widehat{\We}$
\begin{equation*}
\widehat{\We}(\bFe) = \mathbb W_{\rm conv}(\bFe,\det\bFe) := \frac{\mu}{2}|\bFe|^2 + \Gamma(\det\bFe),
\end{equation*}
where $\mathbb W_{\rm conv}: \RR^{5} \to \RR$ convex. The map
$(\bP,z) \mapsto R(\bP_{\rm old},z_{\rm old},\bP-\bP_{\rm
  old},z-z_{\rm old})$ is convex as well. Hence, one has that
\begin{equation*}
\inf_{\mathcal A} \mathcal F \le \mathcal F(\bvarphi,\bP,z) \le \liminf_{n\to \infty} \mathcal F(\bvarphi_n,\bP_n,z_n) = \inf_{\mathcal A} \mathcal F.
\end{equation*}
Thus, $(\bvarphi,\bP,z)$ is a minimizer of $\mathcal F$.
\end{proof}

\begin{remark}[Existence for $d \ge 3$]\label{rem:exist_3d}
 A crucial step on the above existence proof is the use of 
\cite[Prop. 3.2.4]{Kruzik-Roubicek} entailing convergence of the
determinant of the deformation. This requires 
$\bvarphi$ in $W^{1,d}(\Om;\RR^d)$, where $d$ is the space
dimension. In the three dimensional case, one hence needs to 
strengthen the coercivity assumptions on $\We$ in order to get such
integrability from the boundedness of the energy.   This can be
accomplished by using the {\it Ogden}-type elasticity law \cite[Sec. 2.4]{Kruzik-Roubicek} 
\begin{equation}
\label{coercivityWe}
\We(\bFe) = c|\bFe|^{\qe} + \Psi(\det \bFe),
\end{equation}
where $\qe \ge d, c=\mu/\qe d^{(\qe-2)/2}$, and $\Psi(\delta)=-\mu\ln(\delta)+\lambda(\delta-1)^2/2-\mu d/\qe$.  In the three dimensional simulation in Subsection
\ref{sec:hole} we used  $\qe = 4$.
\end{remark}

As shown in \cite{MSZ19}, as the time step size $\tau$ tends to zero, incremental solutions converge to trajectories fulfilling  the energy balance for all time, as well as a so-called global stability condition. These two properties qualify such time continuous trajectories as {\it energetic solutions} \cite{MR15,MT04} of the quasistatic evolution problem.

\subsection{Numerical implementation}\label{Sec:Numerics}

In the following, we present a series of numerical tests for the quasistatic incremental problem \eqref{inc_min_3D} both in two and three space dimensions. All simulations are performed with the finite element library NETGEN/NGSolve\footnote{www.ngsolve.org} \cite{Sch97,Sch14}. As NGSolve supports symbolic energy formulations and automatic exact differentiation, we can directly use the formulations \eqref{inc_min_0D} and \eqref{inc_min_3D} together with Newton's method to solve the incremental problems. We mention that the therein arising linearized problems are symmetric. Note in particular that there is no need to compute variations of the energy by hand. Again, we regularize potentials as detailed in Subsection~\ref{Subsec:ngsolve}.

\subsection{Space discretization}\label{sec:FEM}
For the spatial discretization we implement a finite element method \cite{Bra13,Zienk00}. Let $\mathcal{T}$ be a shape-regular triangulation of the body $\Omega\subset\mathbb{R}^d$and let $\mathcal{P}^k(\mathcal{T})$ be the set of piecewise polynomials up to order $k$. We define the following conforming finite element spaces
\begin{align*}
& V^k:= \mathcal{P}^k(\mathcal{T}) \cap C^0(\Omega)\subset \Hone[\Omega],\\
& Q^k:= \mathcal{P}^k(\mathcal{T})\subset \Ltwo[\Omega].
\end{align*}
In the following, we use the decomposition
\begin{equation*}
\bvarphi(\bx)=\bx+\bu(\bx),
\end{equation*} 
where $\bu: [0,T] \times \Om \to \RR^d$ is the displacement. All minimization problems are stated and solved in terms of the displacement field $\bu$ rather than the deformation $\bvarphi$. The displacement $\bu$ and damage $z$ are discretized with $\Hone$-conforming Lagrangian elements of order $k$, i.e.
\begin{align*}
& \bu\in [V^k]^d,\qquad z\in V^{k}.
\end{align*}
For the plastic strain, components are discretized by $\Ltwo$-conforming elements of one polynomial order less than that of the displacement. Here, to satisfy the isochoric condition $\det \bP = 1$ pointwise, in two dimensions the plastic strain is imposed to have the structure
\begin{align}
\label{eq:fes_plastic_strains}
& \bP:= \mat{p_{11} & \,\,\,p_{12} \\ p_{21} & (1+p_{12}p_{21})/p_{11}}\in\text{SL}(2),\qquad p_{11}, p_{12},p_{21}\in Q^{k-1}, \ p_{11}\not =0.
\end{align}
In three dimensions we prescribe
\begin{align}
\label{eq:fes_plastic_strains_3D}
& \bP:= \mat{p_{11} & p_{12} & p_{13} \\ p_{21} & p_{22} & p_{23} \\
p_{31} & p_{32} & p_{33}}\in \text{SL}(3), \qquad p_{ij} \in Q^{k-1}, \ p_{11}p_{22}\not =p_{21}p_{12},
\end{align}
with
\begin{equation}
p_{33} :=\frac{1+p_{23}(p_{11}p_{32}-p_{12}p_{31})+p_{13}(p_{31}p_{22}-p_{21}p_{32})}{p_{11}p_{22}-p_{21}p_{12}}.
\end{equation}
Note that this parametrization of the plastic strain $\bP$ is not global, as no global parametrization $\RR^{d^2-1} \to \SL$ is available. However, in the frame of our approximation we compute $\bP =
\delta\bP\bP_{\rm old}$ with increments $\delta\bP \in \SL$. In fact, by choosing the time step $\tau$
small, we expect and check that $\delta\bP $ stays close to the
identity. As the parametrizations
\eqref{eq:fes_plastic_strains}-\eqref{eq:fes_plastic_strains_3D} are well defined around the identity matrix we use $\delta\bP$ as unknown instead of $\bP$.

\section{Simulations}\label{Sec:Examples}
In the following we present several numerical experiments for the quasistatic problem \eqref{inc_min_3D} in 2D and 3D. We record also code snippets in order to emphasize the flexibility of the NGSolve implementation.
\begin{table}[h!]
	\begin{tabular}{cccc}
		$\sigma_{\rm z}$ [MPa] & $\rho_0$ & $\zeta_0$ & $\muz$ \\
		\hline 
		$2/3$ & 0.1 & 0.5 & $10^{-4}$
	\end{tabular}
\caption{Damage yield stress, damage coupling parameters, and damage length-scale coefficient used for numerical examples.}
\label{tab:material_parameters_2D}
\end{table}
The material parameters for the following examples are chosen as in Table~\ref{tab:material_parameters_2D}.
The Young's modulus $E$, the Poisson's ratio $\nu$, the plastic yield stress $\sigma_{\rm p}$, and the hardening parameter $H$ are the same in Table~\ref{tab:material_parameters}.

\subsection{2D plate with hole without damage}\label{Subsec:example_without_dam}
In the first 2D example, a square-shaped plate with a circular hole at $m=(m_x,m_y)=(1/4,3/4)$ in the middle of the upper left quarter, see Figure~\ref{fig:plate_hole_geometry}, is fixed on the left side of its boundary and a time-dependent traction force $\boldsymbol{g}$ is applied on the right side. The maximal applied force is $g_{\rm max} = 340$. In this example we neglect damage ($z \equiv 1$). The geometry parameters can be found in Table~\ref{tab:plate_hole_geom_par}.

\begin{table}[h!]
	\begin{tabular}{ccccc}
		$l$ [m] & $b$ [m] & $r$ [m] & $m_x$ [m] & $m_y$ [m]\\
		\hline 
		1 & 1 & 0.1 & 0.25 & 0.75
	\end{tabular}
	\caption{Geometry parameters for 2D plate with hole.}
	\label{tab:plate_hole_geom_par}
\end{table}

\begin{figure}[h!]
	\centering
	\includegraphics[width=0.275\textwidth]{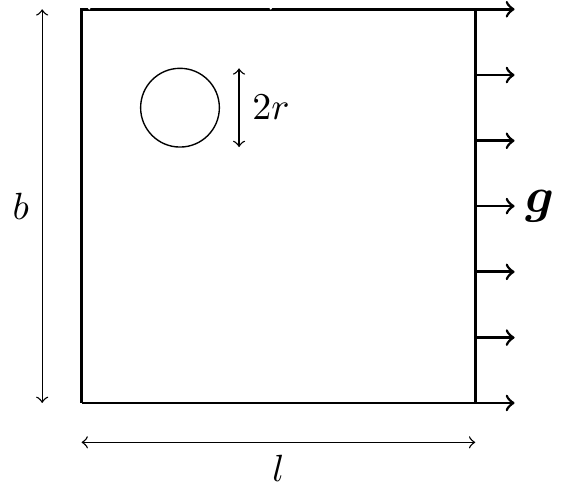}
	\caption{Geometry for plate with hole.}
	\label{fig:plate_hole_geometry}
\end{figure}

\begin{lstlisting}[language=Python, firstline=4, lastline=8,caption={Utility functions.},label=lst:utlityfct]
def RegNorm(mat, eps):
    return sqrt(InnerProduct(mat,mat) + eps**2) - eps

def NeoHooke(F, mu, lam):
    return mu/2*(InnerProduct(F,F)-mesh.dim) - mu*log(Det(F)) + lam/2*(Det(F)-1)**2
\end{lstlisting}

In Listing~\ref{lst:utlityfct} one can see utility functions for the regularized Frobenius scalar product and the material law of Neo--Hooke \eqref{eq:neo_hooke_en}. The latter is given in terms of the deformation gradient $\bF$ and the Lam{\'e} coefficients $\mu$ and $\lambda$, cf. \eqref{eq:neo_hooke_en}.

\begin{lstlisting}[language=Python, firstline=57, lastline=65, caption={Finite element spaces and trial functions for finite plasticity in 2D.},label=lst:fesplast]
# Finite element spaces
fesH1 = VectorH1(mesh, order=k, dirichlet="left")
fesL2 = L2(mesh, order=k-1)
fesP  = FESpace( [fesL2,fesL2,fesL2] )
fes   = FESpace( [fesH1,fesL2,fesL2,fesL2] )
# Trialfunctions
u, dP11, dP12, dP21 = fes.TrialFunction()
dP = CoefficientFunction( (dP11, dP12, dP21, (1+dP12*dP21)/dP11), dims=(2,2) )
\end{lstlisting}

In Listing~\ref{lst:fesplast} the finite element spaces and the corresponding trial functions used for the symbolic formulation are defined. In Appendix~\ref{Sec:gen_meshes}, examples of mesh generation with NETGEN/NGSolve in two and three dimensions can be found. The components of the displacement $\bu=(u_1,u_2)$ belong to $\Hone$-conforming finite element spaces of a given polynomial order
$k$. As mentioned in Subsection~\ref{sec:FEM}, the plastic strains are approximated in a $\Ltwo$-conforming finite element space with one polynomial order less. In line 4 of Listing~\ref{lst:fesplast} a so-called compound finite element space with three copies of $\Ltwo$ is created. This is needed for the plastic strains having three independent components in 2D. In line 5 a compound finite element space with the displacements and again three copies of $\Ltwo$-space is defined and in line 7 the \texttt{TrialFunctions} for the displacements and the (multiplicative) plastic update are created. The plastic strain $\delta \bP$ is defined in line 8 according to equation \eqref{eq:fes_plastic_strains}. Notice that matrices in NGSolve are defined row-wise.

\begin{lstlisting}[language=Python, firstline=66, lastline=75, caption={\texttt{GridFunctions} and \texttt{CoefficientFunctions} for finite plasticity in 2D.},label=lst:gfcfplast]
# Gridfunctions
P_old  = GridFunction(fesP)
P_tmp  = GridFunction(fesP)
gf_sol = GridFunction(fes)
gf_u, gf_dP11, gf_dP12, gf_dP21 = gf_sol.components
# Plastic strains (update, old, new)
cf_dP    = CoefficientFunction( (gf_dP11,gf_dP12,gf_dP21,(1+gf_dP12*gf_dP21)/gf_dP11), dims=(2,2) )
cf_P_old = CoefficientFunction( (P_old.components,(1+P_old.components[1]*P_old.components[2])/P_old.components[0]), dims=(2,2) )
cf_P = cf_dP*cf_P_old
P    = dP*cf_P_old
\end{lstlisting}

To save the results, so-called \texttt{GridFunctions} are used, containing the finite element coefficient vectors, see lines 2-4 in Listing~\ref{lst:gfcfplast}.
Again, it is useful to work with matrices and to define the new plastic strain $\bP$ in terms of the update $\delta \bP$ (\texttt{dP} in the code) and the old plastic strains $\bP_{\rm old}$ (\texttt{P\_old} in the code). Note, that \texttt{cf\_P} is a \texttt{GridFunction} object, whereas \texttt{P} is a \texttt{TrialFunction} used in the variational formulation.

\begin{lstlisting}[language=Python, firstline=78,
lastline=80,caption={Symbolic definition of tensors.},label=lst:symbdefplast]
I  = Id(mesh.dim)
F  = I + Grad(u)
Fe = F*Cof(P).trans
\end{lstlisting}

Definitions of the deformation gradient $\bF$ gives
strain $$\bFe = \nabla \bvarphi \bP^{-1} = (\bI + \nabla
\bu)\bP^{-1} = (\bI + \nabla \bu)(\cof \bP)^\top,$$ where we used the relation $\bP^{-1} = (\cof \bP)^\top$ from $\det \bP =1$, see Listing~\ref{lst:symbdefplast}.

\begin{lstlisting}[language=Python, firstline=81,
lastline=86,caption={Bilinear form for finite plasticity.},label=lst:blfplast]
# Definition of the Energy 
a = BilinearForm(fes, symmetric=True)
a += Variation((NeoHooke(Fe,mu,lam) + H/2*InnerProduct(P-I,P-I))*dx)
a += Variation(sigma_p*RegNorm(dP-I,eps)*dx)
a += Variation(-factor*force*u*ds("right"))
\end{lstlisting}

In Listing~\ref{lst:blfplast} we define a \texttt{BilinearForm} on the compound finite element space \texttt{fes} implementing minimization problem \eqref{inc_min_3D} (without damage). The resulting matrices during Newton iterations are symmetric, saving memory space.
In lines 3-5 volume (\texttt{dx}) and boundary (\texttt{ds}) energy integrators are added, where \texttt{CoefficientFunctions} and \texttt{TrialFunctions} can be freely combined. The \texttt{definedon} flag is used to prescribe a traction force on the right boundary. The \texttt{force} is also a \texttt{CoefficientFunction} and a parameter \texttt{factor} is used to implement an increasing or decreasing external force $\bg = (g_1,0)$ in time. 

\begin{lstlisting}[language=Python, firstline=54, lastline=55]
factor = Parameter(0.0)
force  = CoefficientFunction( (340,0) )
\end{lstlisting}

\begin{lstlisting}[language=Python, firstline=87, lastline=88,caption={Optimizations in the bilinear form.},label=lst:optblf]
a = BilinearForm(fes, symmetric=True, condense=True)
a += Variation((...).Compile()*dx)
\end{lstlisting}

Some further improvements can be implemented, cf. Listing~\ref{lst:optblf}. First, all degrees of freedom regarding the plastic strains are discontinuous and can be eliminated at the element level. Therefore, the flag \texttt{condense} is used, prescribing the \texttt{BilinearForm} to compute the corresponding Schur complement and harmonic extension for the static condensation. Note that adding symbolic expressions to an integrator is practical, but a complicated expression tree can slow down the assembling process. A use of the \texttt{Compile} function as in line 2 of Listing~\ref{lst:optblf} optimizes the expression tree.
\begin{align*}
\text{\texttt{Compile(truecompile=False, wait=False)}}
\end{align*}
Setting the first argument to \texttt{True}, the Python code gets precompiled and then linked. Thus, the effectiveness of the generated code is comparable with integrators written directly in C++.

\begin{lstlisting}[language=Python, firstline=101, lastline=126, caption={Time loop and update process for finite plasticity.},label=lst:timeloopplast]
# Set dP = P_old = I
gf_dP11.Set(1)
P_old.components[0].Set(1)

with TaskManager():
    while t < g_max:
        # Increase time-step
        t += tau
        factor.Set(t)
        
        # Reset dP = I
        gf_sol.components[1].Set(1)
        gf_sol.components[2].Set(0)
        gf_sol.components[3].Set(0)
        
        # Newton minimization
        (res,numit) = solvers.NewtonMinimization(a, gf_sol, maxit=100, maxerr=1e-8, linesearch=True, printing=False, inverse="sparsecholesky")
        
        # Update plastic strains
        P_tmp.components[0].Set(cf_P[0])
        P_tmp.components[1].Set(cf_P[1])
        P_tmp.components[2].Set(cf_P[2])
        
        P_old.components[0].vec.data = P_tmp.components[0].vec
        P_old.components[1].vec.data = P_tmp.components[1].vec
        P_old.components[2].vec.data = P_tmp.components[2].vec
\end{lstlisting}

\begin{figure}[h!]
	\centering
	\includegraphics[width=0.49\textwidth]{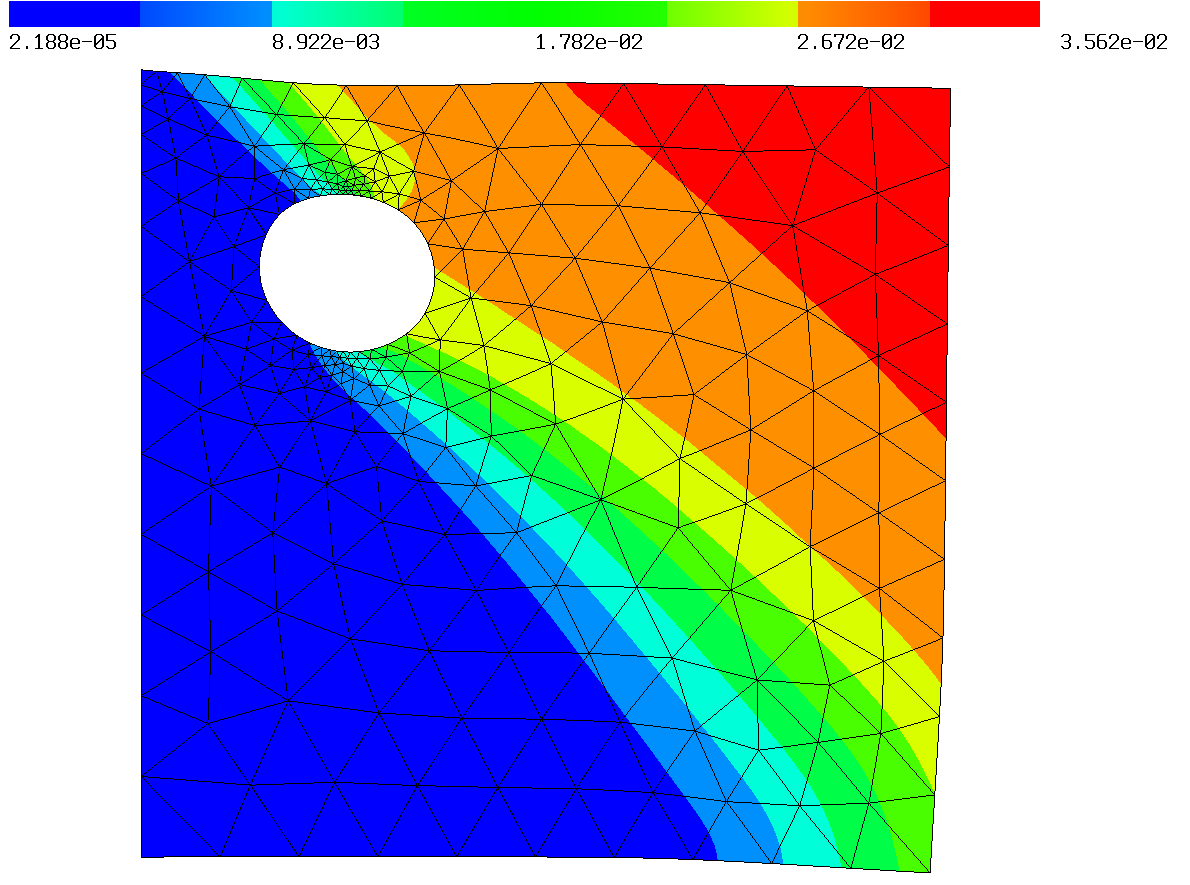}
	\includegraphics[width=0.49\textwidth]{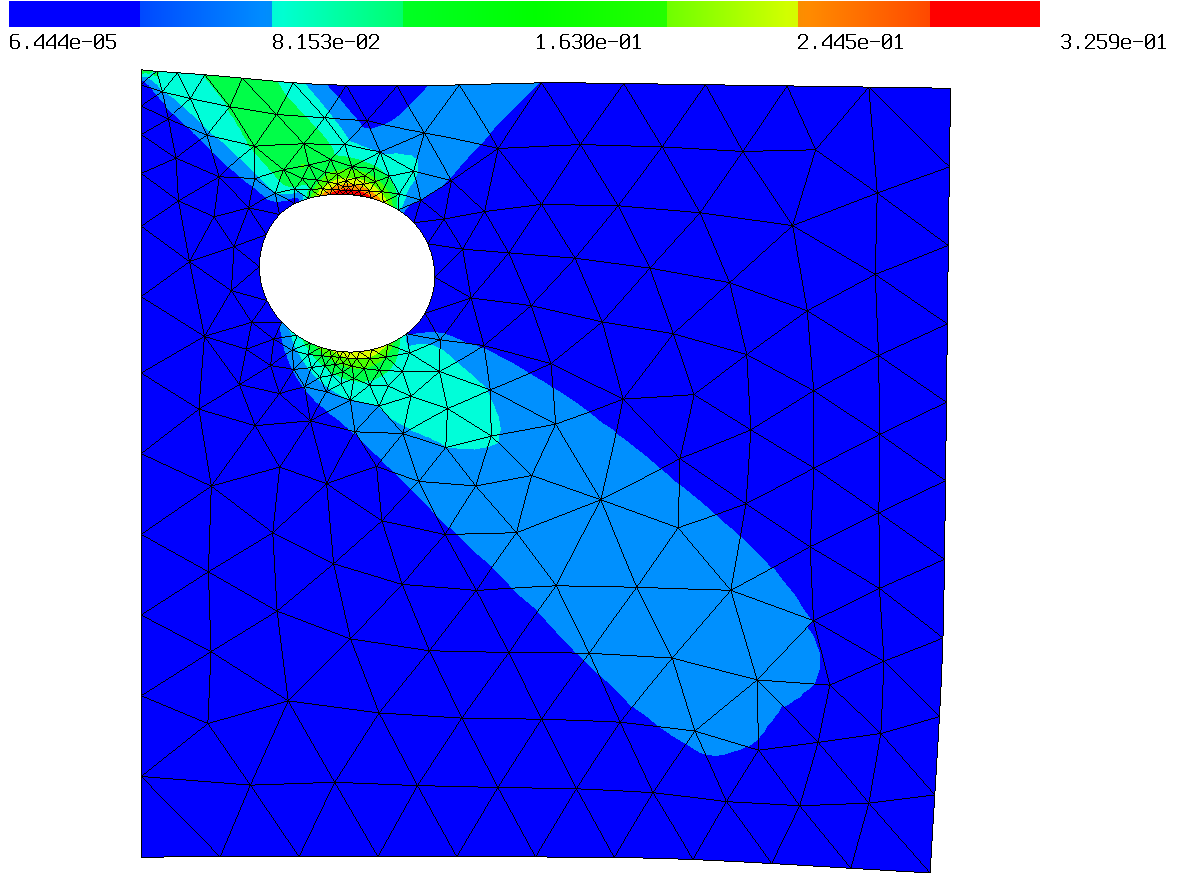}
	\caption{Norm of displacement $\bu$ (left) and plastic strain $\bP$ (right) at maximal traction force $g_{\mathrm{max}}=340$, time step size $\tau=5\cdot 10^{-5}$, and polynomial order $k=3$.}
	\label{fig:result_2d_nodam}
\end{figure}

Before starting the time stepping, the plastic strain is initialized as $\bP=\bP_{\rm old}=\bI$. With the \texttt{Set} method the \texttt{CoefficientFunction} 1 gets interpolated into the finite element space. Since the default value for a \texttt{GridFunction} is zero, all other variables do not need to be initialized.

During the time stepping, see Listing~\ref{lst:timeloopplast}, the update plastic strain $\delta \bP$ is reset to the identity matrix, whereas the displacement is held in order to provide good initial guesses for Newton's method. The Newton minimization algorithm receives the \texttt{BilinearForm} and a \texttt{GridFunction} where the solution is saved. With \texttt{linesearch=True} an optional line-search algorithm is used, where the energy is computed in every Newton step and energy decrease is guaranteed at every step. The built in sparsecholesky solver is used for inverting. After convergence, one needs to update the plastic strain, as only the update $\delta \bP$ has been computed. Therefore, we set the expression for the new plastic strain (see line 9 in Listing~\ref{lst:gfcfplast}) into a temporary \texttt{GridFunction}. Then, the old plastic strain is overwritten. Newton's method recognizes if the condense flag in the \texttt{BilinearForm} is set and acts accordingly. The command in line 5 in Listing~\ref{lst:timeloopplast} activates the \texttt{TaskManager} of NGSolve to assemble the finite element matrices in parallel.

To implement a standard Newton solver one needs to differentiate the energy formulation once to obtain the residual and twice for the linearization matrix. Both tasks are performed automatically in NGSolve by the \texttt{Apply} and \texttt{AssembleLinearization} method of the \texttt{BilinearForm}, see Listing~\ref{lst:newton}. This can be extended by a line-search algorithm.

\begin{lstlisting}[language=Python,caption={Newton's method in NGSolve.},label=lst:newton]
r = gfu.vec.CreateVector()
w = gfu.vec.CreateVector()
    
for it in range(maxits):
    a.Apply(gfu.vec, r)
    a.AssembleLinearization(gfu.vec)
    
    inv = a.mat.Inverse(gfu.space.FreeDofs(), inverse=inverse)
    w.data = inv * r
    u.vec.data -= w
    
    if sqrt(abs(InnerProduct(w,r))) < tol:
        break
\end{lstlisting}

\subsection{2D plate with hole with damage}\label{Subsec:example_with_dam}
With respect to the previous example, here damage is also taken into account.

\begin{lstlisting}[language=Python, firstline=1,
lastline=8,caption={Dissipation-coupling and energy-degradation functions for damage.},label=lst:utilityfctdam]
def Psi(z, z_old, eps):
    return sigma_z*IfPos(z-z_old+eps, (z-z_old+eps)**3/(3*eps**2) + z_old-z, z_old-z)

def Rho(z):
    return rho0 + (1-rho0)*(IfPos(z, z, 0))**2

def Zeta(z):
    return zeta0 + (1-zeta0)*(IfPos(z, z, 0))**2
\end{lstlisting}
In Listing~\ref{lst:utilityfctdam} three additional utility functions are shown: The damage dissipation $R_{\rm z}^\varepsilon$, see \eqref{def:Psiepsilon}, the coupling function $\rho$, see \eqref{def:rho}, and the elastic degradation function $\zeta$, see \eqref{def:zeta}. 

\begin{lstlisting}[language=Python, firstline=11, lastline=20,caption={Finite element spaces and trial functions for finite plasticity with damage.},label=lst:fesdam]
# Finite element spaces
fesH1 = VectorH1(mesh, order=k, dirichlet="left")
fesL2 = L2(mesh, order=k-1)
fesZ  = H1(mesh, order=k)
fes_P = FESpace( [fes2,fes2,fes2] )
fes   = FESpace( [fesH1,fesL2,fesL2,fesL2, fesZ] )
# Trialfunctions
u, dP11,dP12,dP21, z = fes.TrialFunction()
dP = CoefficientFunction( (dP11, dP12, dP21, (1+dP12*dP21)/dP11), dims=(2,2) )
\end{lstlisting}

The damage $z$ is discretized by a $\Hone$-conforming finite element space with the same polynomial degree as the displacement, see Listing~\ref{lst:fesdam}. 

\begin{lstlisting}[language=Python, firstline=35,
lastline=47,caption={Bilinear form for finite plasticity coupled with damage.}, label=lst:blfdam]
a = BilinearForm(fes, symmetric=True)
a += Variation((Zeta(z)*NeoHooke(Fe,mu,lam) + 0.5*H*InnerProduct(P-I,P-I))*dx)
a += Variation((Psi(z,z_old,eps) + Rho(z_old)*sigma_p*(RegNorm(dP-I,eps)))*dx)
a += Variation(mu_z/2*grad(z)*grad(z)*dx)
a += Variation(-factor*force*u*ds("right"))

# Set dP = P_old = I
gf_dP11.Set(1)
P_old.components[0].Set(1)
# Set z = z_old = 1
gf_z.Set(1)
z_old.Set(1)
\end{lstlisting}

The corresponding \texttt{BilinearForm} has the same structure as in Listing~\ref{lst:blfplast}, however, the functions $\zeta$, $\rho$, and the dissipation function $R_{\rm z}^\varepsilon$ are incorporated. At the initial time, the material is assumed to be undamaged, namely $z$ is set to $1$, see Listing~\ref{lst:blfdam}. 

\begin{lstlisting}[language=Python, firstline=49,lastline=58,caption={Decoupling $(\bu,\bP)$ and $z$.}, label=lst:decouple]
freedofs_uP[fes.Range(4)] = False
freedofs_z = BitArray(fes.FreeDofs(a.condense))
for i in range(4):
    freedofs_z[fes.Range(i)] = False
    
...

solvers.NewtonMinimization(a, gfsol, freedofs=freedofs_uP, linesearch=True, inverse="sparsecholesky")
solvers.NewtonMinimization(a, gfsol, freedofs=freedofs_z, linesearch=True, inverse="sparsecholesky")
\end{lstlisting}

We first solve with respect to $(\bu,\bP)$ and then with respect to $z$. Therefore, the degrees of freedom have to be set up properly, see Listing~\ref{lst:decouple}. In highly damaged elements, namely $z\approx 0$, Newton's method may converge slower. It is possible to stop the evolution of damage at the corresponding elements by locking the degrees of freedom.

\begin{lstlisting}[language=Python,caption={Setting pre-damage.}, label=lst:predam]
# Set z = z_old = 1 for non-damaged area, 0.1 for pre-damaged area
SetDamage(gf_z, undamaged=1, damaged=0.1)
SetDamage(z_old, undamaged=1, damaged=0.1)
\end{lstlisting}

\begin{figure}[h!]
	\centering
	\includegraphics[width=0.275\textwidth]{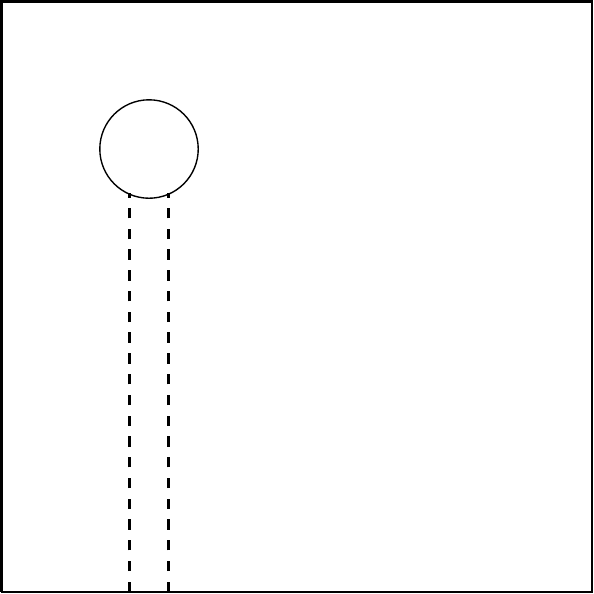}
	\caption{Geometry for plate with hole with pre-damaged area (inside dashed lines).}
	\label{fig:plate_hole_predam}
\end{figure}

In case of a pre-damaged material, see Figure~\ref{fig:plate_hole_predam}, one has to specify the pre-damaged area inside the geometry, see Appendix~\ref{Sec:gen_meshes}. The initial damage values are set as in Listing~\ref{lst:predam}.

\begin{figure}[h!]
	\centering
	\includegraphics[width=0.49\textwidth]{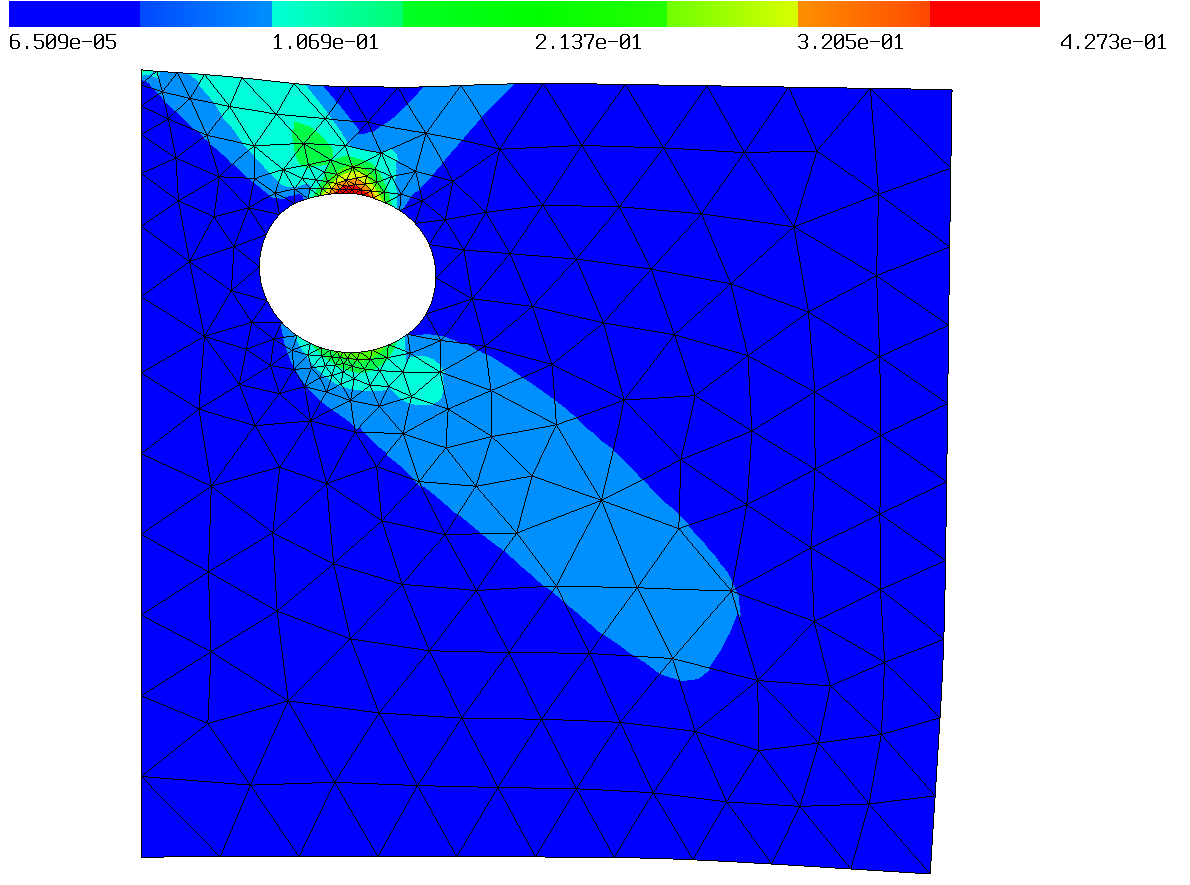}
	\includegraphics[width=0.49\textwidth]{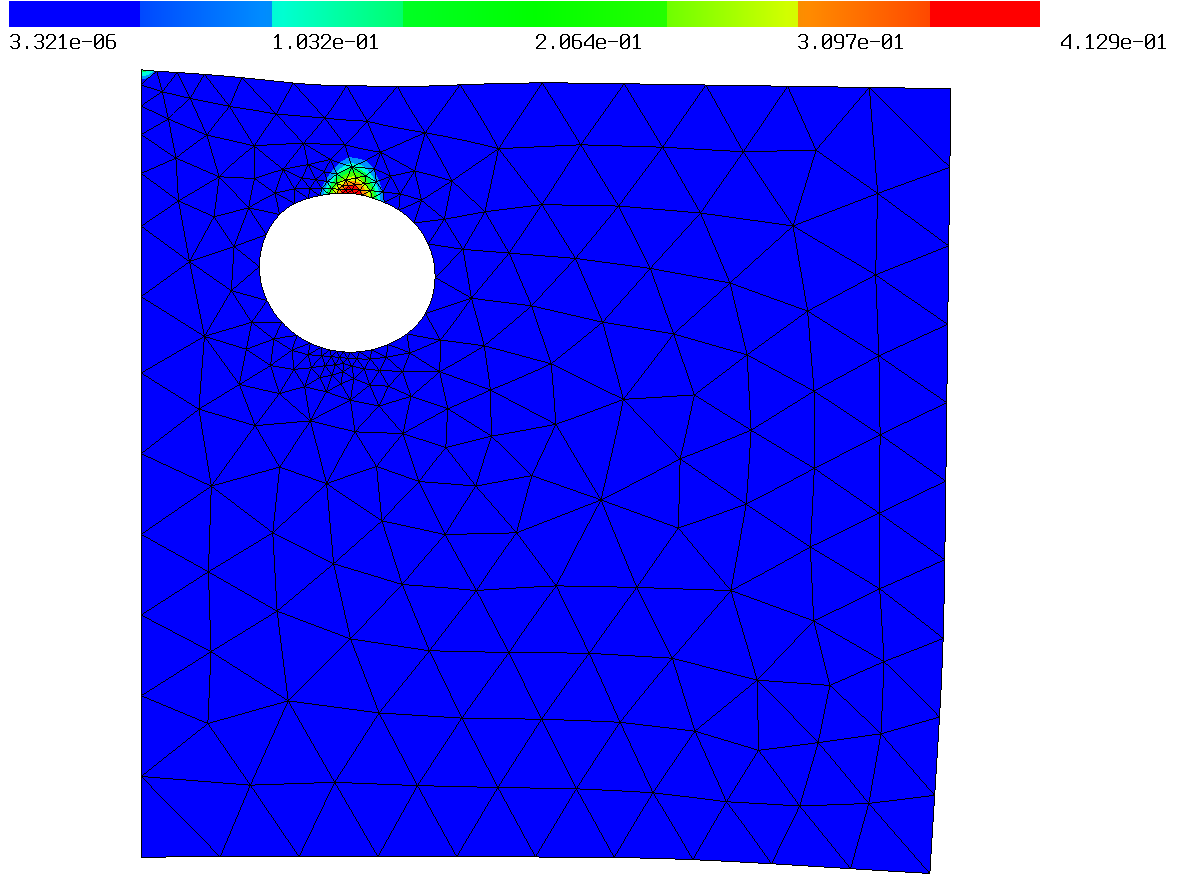}
	\caption{Norm of plastic strain $\bP$ (left) and damage $1-z$ (right) for $g_{\rm max}=340$, $\tau=5\cdot 10^{-5}$, and $k=3$ at maximal traction force.}
	\label{fig:result_2d_dam}
\end{figure}

\begin{figure}[h!]
	\centering
	\includegraphics[width=0.49\textwidth]{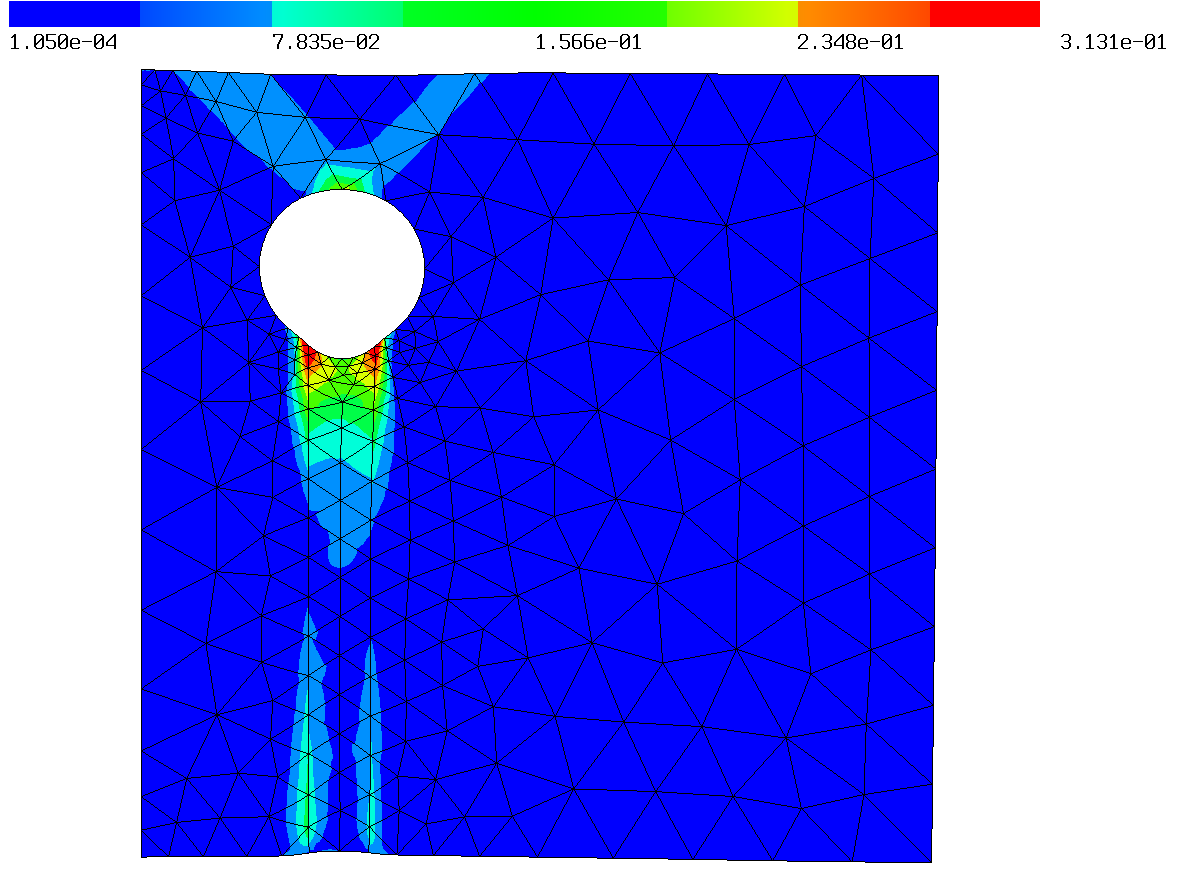}
	\includegraphics[width=0.49\textwidth]{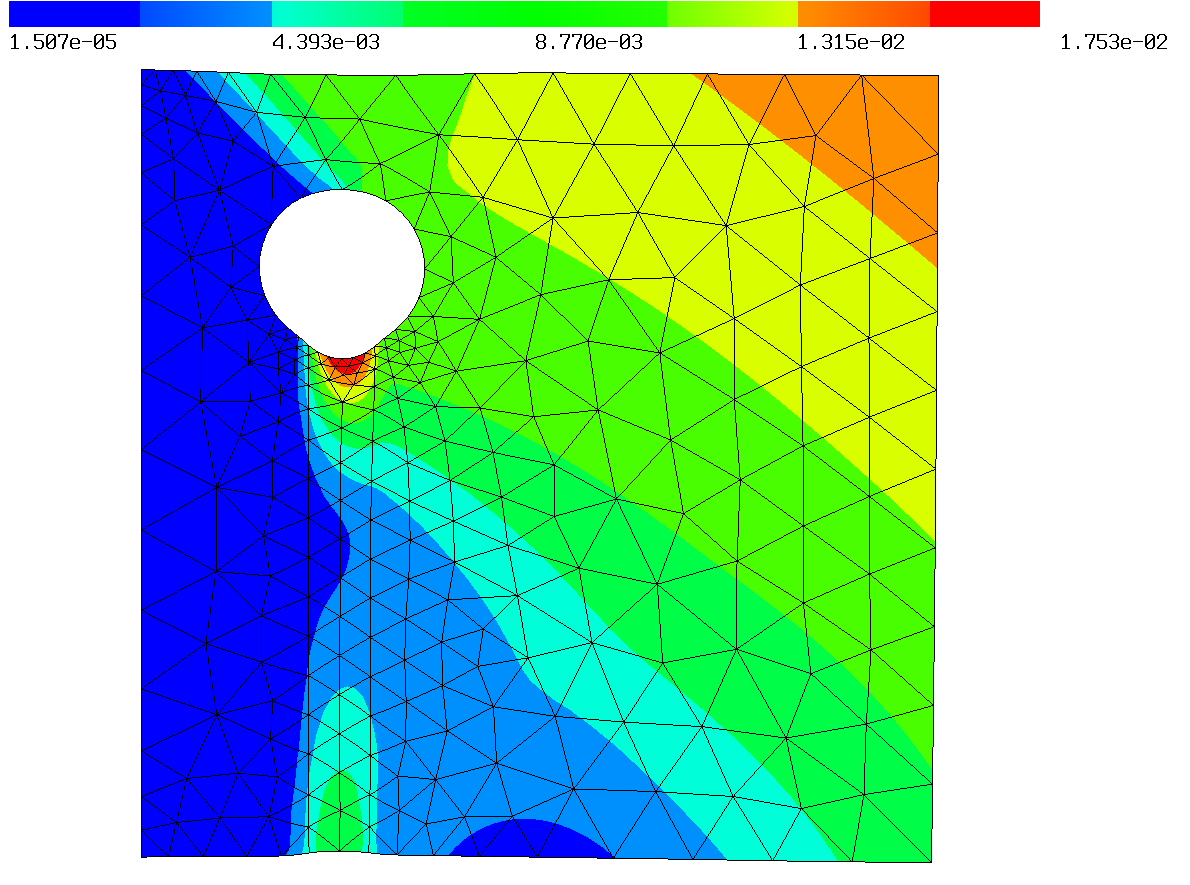}
	\caption{Norm of plastic strain $\bP$ (left) and displacement $\bu$ (right) for $g_{\rm max}=270$, $\tau=5\cdot 10^{-5}$, and polynomial order $k=3$ and pre-damaged area at maximal traction force.}
	\label{fig:result_2d_predam}
\end{figure}

\subsection{2D plate with hole with time-dependent Dirichlet condition}
In this section the setting of the first example is used, but two time-dependent displacements are prescribed on the right boundary instead of a traction force:
\begin{align}
\label{time-dependent-Dir}
\bu^1_D(t) =
\mat{0.1\,t\\0},\qquad {\bu^2_D}(t) = \mat{0.1\,t\\\text{free}}.
\end{align}

\begin{figure}[h!]
	\centering
	\includegraphics[width=0.49\textwidth]{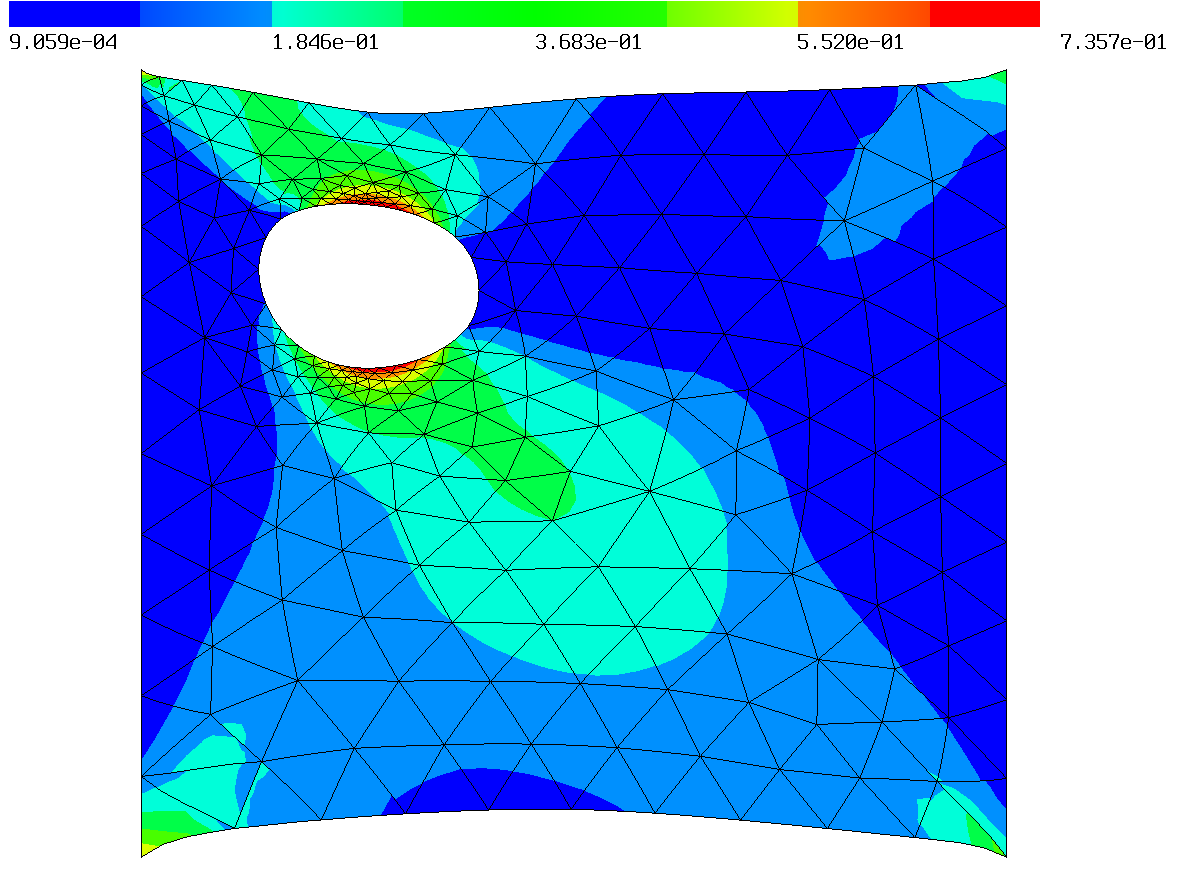}
	\includegraphics[width=0.49\textwidth]{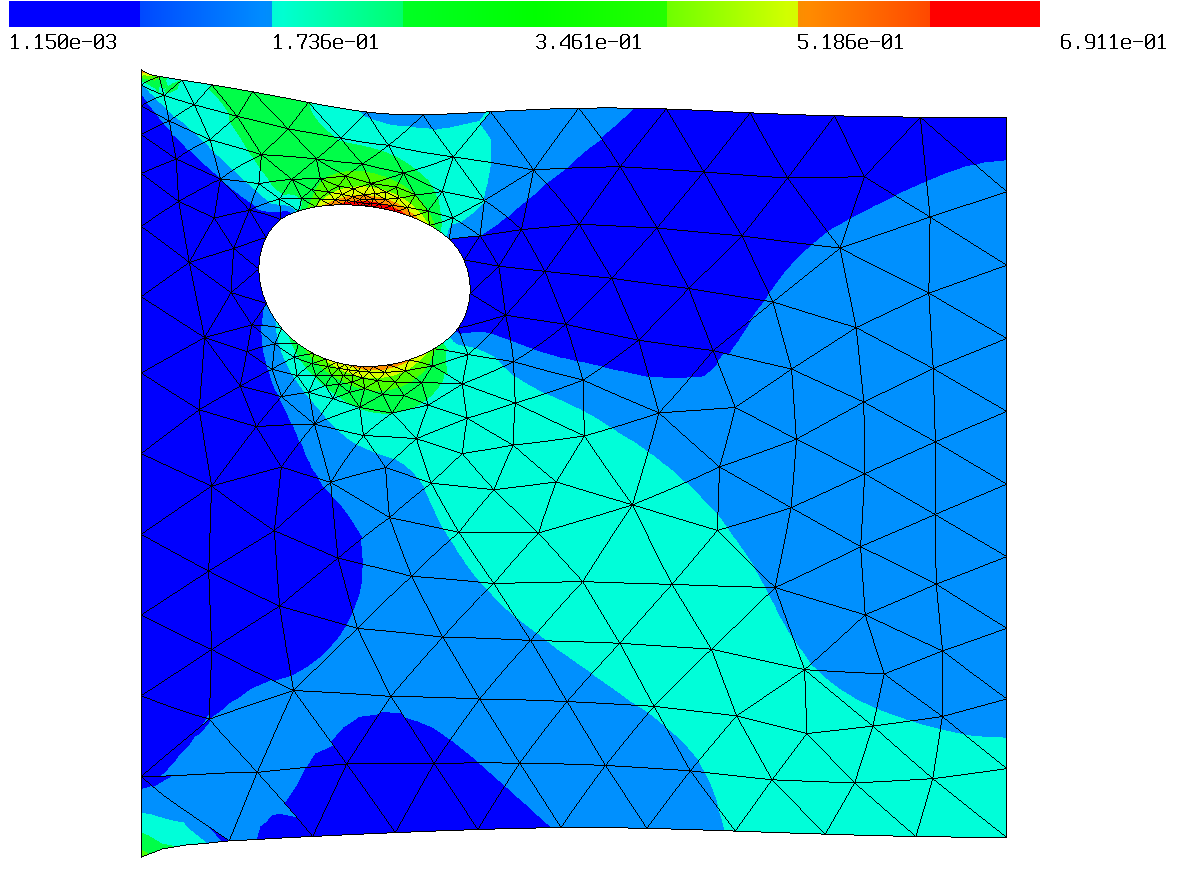}
	
	\caption{Norm of plastic strain $\bP$ for $\bu^1_D$ (left) and $\bu^2_D$ (right) with $\tau=10^{-5}$ and polynomial order $k=3$ at maximal boundary displacement.}
	\label{fig:result_dir_pen_predam}
\end{figure}

\begin{lstlisting}[language=Python,caption={Dirichlet data $\bu^1_D$ as penalty formulation.}, label=lst:dir_pen]
uD = CoefficientFunction( (0.1,0) )

a += Variation(1/eps*InnerProduct(u-factor*uD,u-factor*uD)*ds("right"))
\end{lstlisting}

Here, the prescribed boundary condition is imposed by including
an additional penalization term of the form
\begin{align}
\frac{1}{\varepsilon} \int_{\Gamma_D}|\bu(\bx)-\bu_D(t_i,\bx)|^2\,dS(x)
\end{align}
into \eqref{inc_min_3D}, see Listing~\ref{lst:dir_pen}. For the first choice in \eqref{time-dependent-Dir} both components of $\bu$ are prescribed, leading to large strains at the right corners. For the second choice, the vertical component is free, see Figure~\ref{fig:result_dir_pen_predam}.

\subsection{3D plate with hole}\label{sec:hole}
In this section we present an example in 3D, see
Figure~\ref{fig:plate3d_hole_geometry} and
Table~\ref{tab:plate3d_hole_geom_par}.  Here, we use the elastic
energy  of Ogden-type from \eqref{coercivityWe}, see Listing~\ref{lst:ogden3d}. The construction of the
corresponding mesh can be found in Appendix~\ref{Sec:gen_meshes}. The
implementation only needs to be adjusted with respect to the 
parametrization  of $\SL$, see equation \eqref{eq:fes_plastic_strains_3D} vs. \eqref{eq:fes_plastic_strains}.
\begin{table}[h!]
	\begin{tabular}{cccccc}
		$l$ [m] & $b$ [m] & $w$ [m] & $r$ [m] & $m_x$ [m] & $m_y$ [m]\\
		\hline 
		1 & 1 & 0.2 & 0.1 & 0.25 & 0.75
	\end{tabular}
	\caption{Geometry data 3D example.}
	\label{tab:plate3d_hole_geom_par}
\end{table}

\begin{figure}[h!]
	\centering
	\includegraphics[width=0.275\textwidth]{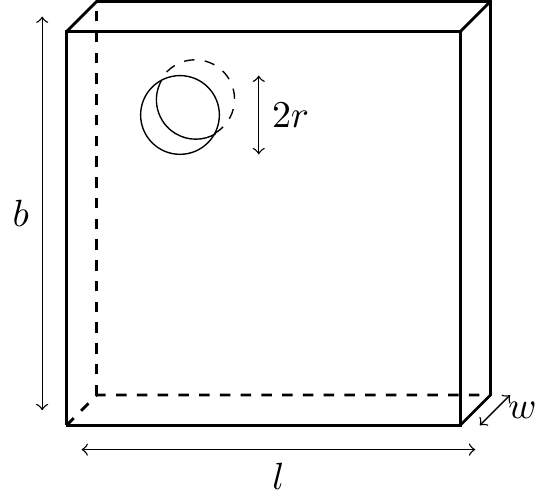}
	\caption{Geometry for 3D-plate with hole.}
	\label{fig:plate3d_hole_geometry}
\end{figure}

\begin{lstlisting}[language=Python, firstline=7, lastline=8,caption={Ogden-type energy.}, label=lst:ogden3d]
def Ogden(F, mu, lam):
    return mu/12*(InnerProduct(F,F)**2-3**2) - mu*log(Det(F)) + lam/2*(Det(F)-1)**2
\end{lstlisting}

\begin{lstlisting}[language=Python, firstline=54, lastline=62,caption={Finite element spaces and trial functions for 3D examples.}, label=lst:fes3d]
# Finite element spaces
fesH1 = VectorH1(mesh, order=k, dirichlet="left")
fesL2 = L2(mesh, order=k-1)
fes_P = FESpace([fesL2,fesL2,fesL2,fesL2,fesL2,fesL2,fesL2,fesL2])
fes   = FESpace([fesH1, fesL2,fesL2,fesL2,fesL2,fesL2,fesL2,fesL2,fesL2])
# Trialfunctions
u, dP11, dP12, dP13, dP21, dP22, dP23, dP31, dP32 = fes.TrialFunction()
dP = CoefficientFunction( (dP11, dP12, dP13, dP21, dP22, dP23, dP31, dP32, (1+dP23*(dP11*dP32-dP12*dP31)+dP13*(dP31*dP22-dP21*dP32))/(dP11*dP22-dP21*dP12)), dims=(3,3) )
\end{lstlisting}

The space $\text{SL}(3)$ has eight independent components, where we again use an embedding which is well-defined around the identity matrix, see Listing~\ref{lst:fes3d} and Subsection~\ref{sec:FEM}. In addition to the finite element spaces and \texttt{TrialFunctions} in Listing~\ref{lst:fes3d}, the \texttt{GridFunctions} are defined respectively, see Listing~\ref{lst:gfcfplast}.
One can completely reuse the symbolic energy formulation from Listing~\ref{lst:blfdam} and also the time stepping changes only marginally, see Listing~\ref{lst:timeloop3d}.

\begin{lstlisting}[language=Python, firstline=100, lastline=116,caption={Time loop for 3D example.}, label=lst:timeloop3d]
while t < g_max:
    # Increase time-step
    t += tau
    factor.Set(t)
    
    # Reset dP = I
    for i in range(1,9):
        gf_sol.components[i].Set( 1 if (i == 1 or i == 5) else 0 )
    
    # Newton minimization
    (res,numit) = solvers.NewtonMinimization(a, gf_sol, maxerr=1e-8, linesearch=True, inverse="sparsecholesky")
    
    # Update plastic strains
    for i in range(8):
        P_tmp.components[i].Set(cf_P[i])
    for i in range(8):
        P_old.components[i].vec.data = P_tmp.components[i].vec
\end{lstlisting}

\begin{figure}[h!]
	\centering
	\includegraphics[width=0.49\textwidth]{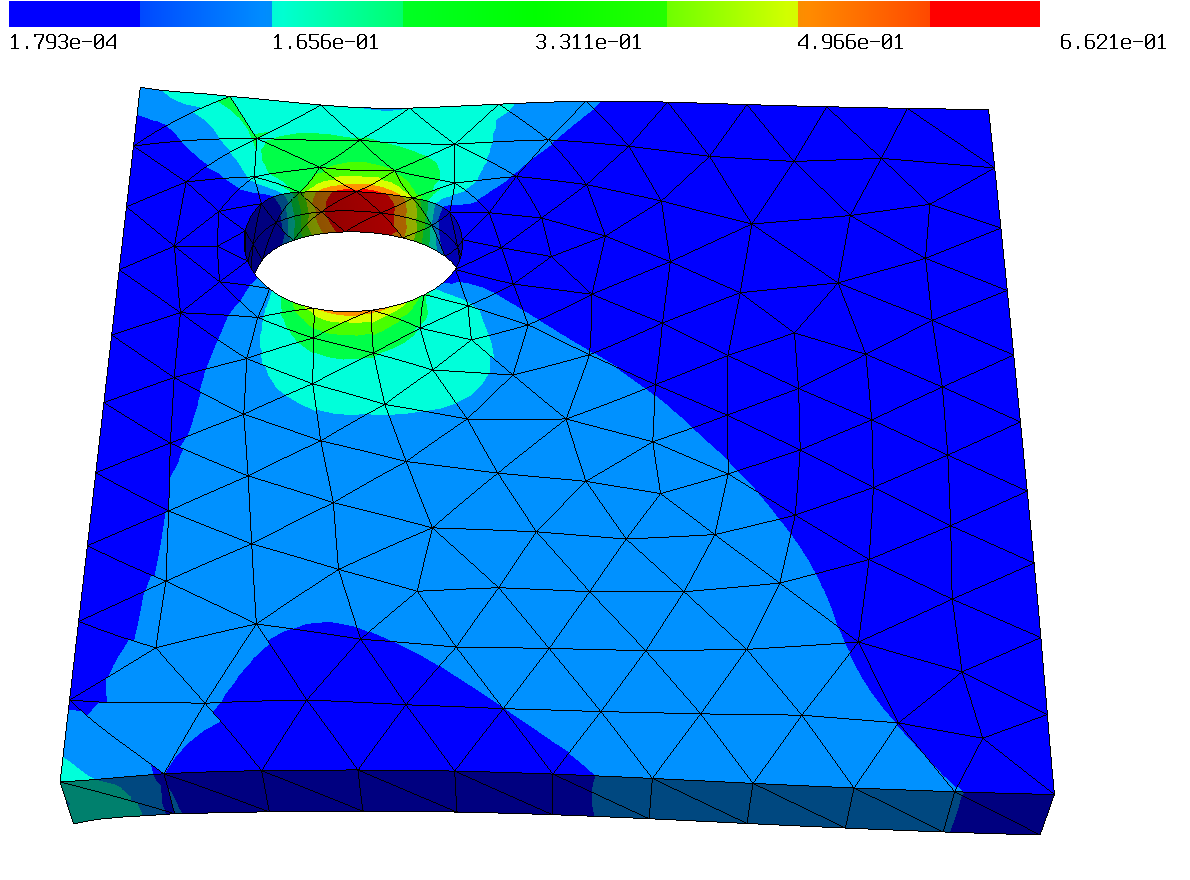}
	\caption{Norm of plastic strain $\bP$ for $g_{\rm max}=340$, $\tau=5\cdot 10^{-5}$, and $k=3$ at maximal traction force.}
	\label{fig:result_3d}
\end{figure}

\section{Conclusions}

We have presented a new phenomenological constitutive model coupling damage and plastic
effects in solids at finite strains. Damage and plastic evolution are rate-independent and coupled
both at the level of the energy and of the dissipation. In particular,
the accumulation of damage induces degradation of the elastic response
and of the effective plastic yield stress. The performance of the
constitutive model has been assessed via a suite of uniaxial and
non-proportional multi-axial  tests.

By coupling the constitutive model with quasistatic equilibrium, the
three dimensional quasistatic incremental problem has been proved to
admit solutions. These are then approximated by a finite element method.
Two- and three-dimensional numerical experiments were performed
within the high-performance multiphysics platform of NETGEN/NGSolve, by
taking advantage of the variational structure of the model. With this
approach, it is possible to work at the level of energy and
dissipation, without calculating derivatives by hand. Different choices of
time-dependent boundary conditions,  as well as the possible occurrence
of pre-damaged regions, can be considered.

Various generalizations of the model could be of interest. Different
choices for elastic
and plastic potentials could be considered.
Damage effects on
plastic hardening, plastic effects on the damage yield stress, and
damage healing could be included, as well as an extension to
viscoplastic response.

\section*{Declaration of interests}
 The authors declare that they have no known competing financial
 interests or personal relationships that could have appeared to
 influence the work reported in this paper.

\section*{Acknowledgments}
\label{sec:acknowledgments}
The support by the Austrian Science Fund (FWF)
projects F\,65, W\,1245,  I\,4354,  and P\,3278 and of the Vienna
Science and Technology Fund (WWTF) project MA14-009 is gratefully
acknowledged.  The
authors are indebted to Alessandro Reali for some interesting comments.

\appendix
\section{Generate meshes}\label{Sec:gen_meshes}
\subsection{2D meshes}
To generate the two dimensional geometry given in Figure~\ref{fig:plate_hole_geometry}, with sizes given in Table~\ref{tab:plate_hole_geom_par}, in NETGEN via NGSolve one can us the \texttt{SplineGeometry} class.

\begin{lstlisting}[language=Python, firstline=1, lastline=20,caption={Gererating a simple 2D geometry in NETGEN via NGSolve.},label=lst:gen2dmesh]
import netgen.geom2d as geom

sg = geom.SplineGeometry()
pnts = [ (0,0,maxh), (L,0,maxh), (L,B,maxh), (0,B,maxh/6), (mx,my-R,maxh/15), (mx+R,my-R,maxh), (mx+R,my,maxh), (mx+R,my+R,maxh), (mx,my+R,maxh/20), (mx-R,my+R,maxh), (mx-R,my,maxh), (mx-R,my-R,maxh) ]
pind = [ sg.AppendPoint(*pnt) for pnt in pnts ]

# Outer rectangle   
sg.Append(['line',pind[0],pind[1]],leftdomain=1,rightdomain=0,bc="bottom")
sg.Append(['line',pind[1],pind[2]],leftdomain=1,rightdomain=0,bc="right")
sg.Append(['line',pind[2],pind[3]],leftdomain=1,rightdomain=0,bc="top")
sg.Append(['line',pind[3],pind[0]],leftdomain=1,rightdomain=0,bc="left")
# Inner circle
sg.Append(['spline3',pind[4 ],pind[5 ],pind[6 ]],leftdomain=0,rightdomain=1,bc="circ")
sg.Append(['spline3',pind[6 ],pind[7 ],pind[8 ]],leftdomain=0,rightdomain=1,bc="circ")
sg.Append(['spline3',pind[8 ],pind[9 ],pind[10]],leftdomain=0,rightdomain=1,bc="circ")
sg.Append(['spline3',pind[10],pind[11],pind[4 ]],leftdomain=0,rightdomain=1,bc="circ")

mesh = Mesh( sg.GenerateMesh(maxh=maxh) )
# Curved elements for circle
mesh.Curve(order)
\end{lstlisting}

First, the points of the geometry are added. As we expect larger plastic strains at the top and bottom of the hole and on the upper left corner, we locally define a finer grid by setting the maximal mesh size at the corresponding points, cf. Listing~\ref{lst:gen2dmesh}.

Next, the lines of the outer rectangle are added and the boundaries are named. With \texttt{leftdomain} and \texttt{rightdomain} one can specify which domain lies left and right of the line, where 0 means no domain. The hole is generated by four spline segments, where the second point denotes the control point of the spline. Then, a mesh is generated automatically by specifying the overall maximal mesh size. To obtain a curved mesh of a specific polynomial order we use the \texttt{Curve} method of the mesh.
The order of curving should coincide with the polynomial order used for the displacement to obtain an isoperimetric mapping. It is also possible to add every element by hand. In addition to triangles, quadrilaterals are also supported by NETGEN.\newline

\begin{lstlisting}[language=Python, firstline=23, lastline=55,caption={Adding damage zone to geometry.}, label=lst:adddamzone]
alpha2 = 3*pi/2-pi/8
alpha3 = 3*pi/2+pi/8
p2 = (R*cos(alpha2)+mx,R*sin(alpha2)+my)
p3 = (R*cos(alpha3)+mx,R*sin(alpha3)+my)
s = -(cos(alpha2)+1)/sin(alpha2)
pm1 = (R*(cos(alpha2)+s*sin(alpha2))+mx,R*(sin(alpha2)-s*cos(alpha2))+my)
pm2 = (-R*(cos(alpha2)+s*sin(alpha2))+mx,R*(sin(alpha2)-s*cos(alpha2))+my)
s = 1/(sin(alpha3)/tan(alpha2)-cos(alpha3))*(sin(alpha2)-sin(alpha3)-(cos(alpha3)-cos(alpha2))/tan(alpha2))
pm3 = (R*(cos(alpha3)+s*sin(alpha3))+mx,R*(sin(alpha3)-s*cos(alpha3))+my)

sg = geom.SplineGeometry()
pnts = [ (0,0),(p2[0],0),(p3[0],0),(L,0),(L,W),(0,W),(p3[0],p3[1]),(pm2[0],pm2[1]),(mx+R,my),(mx+R,my+R),(mx,my+R),(mx-R,my+R),(mx-R,my),(pm1[0],pm1[1]),(p2[0],p2[1]), (pm3[0],pm3[1]) ]
pind = [ sg.AppendPoint(*pnt) for pnt in pnts ]

# Outer rectangle   
sg.Append(['line',pind[0],pind[1]],leftdomain=1,rightdomain=0,bc="bottom")
sg.Append(['line',pind[1],pind[2]],leftdomain=2,rightdomain=0,bc="bottom")
sg.Append(['line',pind[2],pind[3]],leftdomain=1,rightdomain=0,bc="bottom")
sg.Append(['line',pind[3],pind[4]],leftdomain=1,rightdomain=0,bc="right")
sg.Append(['line',pind[4],pind[5]],leftdomain=1,rightdomain=0,bc="top")
sg.Append(['line',pind[5],pind[0]],leftdomain=1,rightdomain=0,bc="left")
# Inner circle
sg.Append(['spline3',pind[6 ],pind[7 ],pind[8 ]],leftdomain=0,rightdomain=1,bc="circ")
sg.Append(['spline3',pind[8 ],pind[9 ],pind[10]],leftdomain=0,rightdomain=1,bc="circ")
sg.Append(['spline3',pind[10],pind[11],pind[12]],leftdomain=0,rightdomain=1,bc="circ")
sg.Append(['spline3',pind[12],pind[13],pind[14]],leftdomain=0,rightdomain=1,bc="circ")
# Damage zone   
sg.Append(['spline3',pind[14],pind[15],pind[6]],leftdomain=0,rightdomain=2,bc="int")
sg.Append(['line',pind[2],pind[6]],leftdomain=2,rightdomain=1,bc="int")
sg.Append(['line',pind[14],pind[1]],leftdomain=2,rightdomain=1,bc="int")

sg.SetMaterial(1, "non_damage")
sg.SetMaterial(2, "damage")
\end{lstlisting}

In Listing~\ref{lst:adddamzone} a second domain is added defining the pre-damaged area of Figure~\ref{fig:plate_hole_predam}. After specifying the damage area one can assign names to the different parts of the material to access them easier.

\subsection{3D meshes}In 3D, the geometry of Figure~\ref{fig:plate3d_hole_geometry} can be created by using Constructive Solid Geometries (CSG).

\begin{lstlisting}[language=Python,caption={Gererating simple 3D geometry in NETGEN via NGSolve.},label=lst:gen3dmesh]
from netgen.csg import *

geo   = CSGeometry()
left  = Plane(Pnt(0,0,0), Vec(-1,0,0)).bc("left")
right = Plane(Pnt(L,B,W), Vec( 1,0,0)).bc("right")
front = Plane(Pnt(0,0,0), Vec(0,-1,0)).bc("front")
back  = Plane(Pnt(L,B,W), Vec(0, 1,0)).bc("back")
bot   = Plane(Pnt(0,0,0), Vec(0,0,-1)).bc("bot")
top   = Plane(Pnt(L,B,W), Vec(0,0, 1)).bc("top")
hole  = Cylinder(Pnt(mx, my, 0), Pnt(mx, my, 1), R).bc("cyl")

cube = left*right*front*back*top*bot
geo.Add(cube-hole)

mesh = Mesh( geo.GenerateMesh(maxh=maxh) )
mesh.Curve(order)
\end{lstlisting}

First, six planes for the sides and an (infinite) cylinder for the hole are created. Then, intersecting the planes and intersecting the complement of the cylinder yields the requested geometry.

\bibliographystyle{acm}
\bibliography{cites}
\end{document}